\documentclass[11pt,onside,onecolumn]{amsart}
 \usepackage[dvips]{epsfig}
 \usepackage{amsgen, amstext,amsbsy,amsopn, amsthm,
amsfonts,amssymb,amscd,amsmat
 h,euscript,enumerate,url,verbatim,calc,xypic}
 \usepackage{textcomp}
 \usepackage{latexsym}
 \usepackage{graphics}
 \usepackage{color}

 \newcommand{\n}{\mathfrak{n} }
 \newcommand{\m}{\mathfrak{m} }

 \newcommand{\p}{\mathcal{P}}

 \newcommand{\Q}{\mathbb{Q}}
 \newcommand{\Z}{\mathbb{Z}}
  \newcommand{\N}{\mathbb{N}}

 \newcommand{\ov}{\overline}

  \newcommand{\Ass}{\operatorname{Ass}}

  \newcommand{\im}{\operatorname{image}}

 \newcommand{\dm}{\operatorname{dim}}
 \newcommand{\h}{\operatorname{ht}}

 \newcommand{\grade}{\operatorname{grade}}
 \newcommand{\depth}{\operatorname{depth}}

 \newcommand{\vol}{\operatorname{vol}}
  \newcommand{\lm}{\lambda}

 \newtheorem{theorem}{Theorem}[section]
 
 \newtheorem{corollary}[theorem]{Corollary}
 \newtheorem{lemma}[theorem]{Lemma}
 \newtheorem{proposition}[theorem]{Proposition}

 \newtheorem{conjecture}{Conjecture}

 \theoremstyle{definition}

 \newtheorem{example}[theorem]{Example}
 \theoremstyle{remark}

\begin{document}

\title[ ] {Normal Hilbert Polynomials : A Survey}
\author[]{Mousumi Mandal, Shreedevi Masuti and J. K. Verma}
\address{Department of Mathematics,Indian Institute of Technology Bombay, Powai,
Mumbai 400076 INDIA}

\email{mousumi@math.iitb.ac.in}
\email{shreedevi@math.iitb.ac.in}
\email{jkv@math.iitb.ac.in}

\maketitle
\begin{abstract} We  survey  some of the  major
results about normal Hilbert polynomials of ideals.
We discuss a formula due to Lipman for complete ideals
in regular local rings of dimension two, theorems of
Huneke, Itoh, Huckaba, Marley and  Rees in Cohen-Macaulay analytically
unramified local rings.
We also discuss recent works of Goto-Hong-Mandal and
Mandal-Singh-Verma concerning the positivity of the first coefficient of the
normal Hilbert polynomial in unmixed analytically unramified local rings.
Results of Moral{\'e}s and Villarreal
linking normal Hilbert polynomial of monomial ideals with Ehrhart polynomials of
polytopes are also presented. 
 \end{abstract}

\thispagestyle{empty}

\tableofcontents

\section*{Introduction}
In this paper we survey results about  normal Hilbert polynomials of ideals
in local Noetherian rings and polynomial rings. We first set-up the notation
and then describe the contents of this paper.   
Let $I$ be an $\m$-primary ideal of a local ring $(R,\m)$ of dimension $d.$ A sequence of
ideals $\mathcal I=\{I_n\}_{n\in \mathbb Z}$ is called an {\bf $I$-admissible 
filtration} if there exists  a $k\in \mathbb N$ such that for all $m,n \in
\mathbb Z$,
 $${\rm (i)}\;\; I_{n+1}\subseteq I_n,
 {\rm (ii)}\;\;  I_m I_n\subseteq I_{m+n} \;\; \text{and}\; {\rm (iii)}\;\;
   I^n\subseteq I_n \subseteq I^{n-k}.$$
   
 Marley in \cite{m} showed that if $\mathcal I$ is an $I$-admissible filtration
then the {\bf Hilbert  function} of $\mathcal I$ defined by  $H_{\mathcal
I}(n)=\lm(R/I_n)$ where $\lm$ denotes length as $R$-module coincides with a
polynomial $P_{\mathcal I}(x)\in \mathbb Q[x]$ of degree $d$ for large $n$. This
polynomial is written as 
 $$P_{\mathcal I}(x)=e_0(\mathcal I){x+d-1\choose d}-e_1(\mathcal{
I}){x+d-2\choose d-1}+\cdots +(-1)^de_d(\mathcal I)$$
and it is called the {\bf Hilbert polynomial of  $\mathcal I$}. The uniquely
determined integers $e_i(\mathcal I)$ for $i=0,1,\ldots ,d$ are called the {\bf
Hilbert coefficients of \;$\mathcal I.$} The coefficient $e_0(\mathcal I)$ is a
positive integer and it is called
the {\bf multiplicity of $\mathcal I.$} The coefficient $e_1(\mathcal I)$ is
called
the {\bf Chern number of $\mathcal I.$ } If $\mathcal I$ is the $I$-adic
filtration, then  $e(I):=e_0(\mathcal I)=e_0(I)$ (resp. $e_1(\mathcal
I):=e_1(I)$) is called the {\bf multiplicity}  (resp. the {\bf Chern number}) of
$I.$

We set $$n(\mathcal I)=\sup \{n\in \mathbb Z\mid H_{\mathcal I}(n)\not=
P_{\mathcal I}(n)\}.$$ 
The integer $n(\mathcal I)$ is called the {\bf postulation number} of $\mathcal
I$.
 For an ideal $I$ in a ring $R$, the {\bf integral closure} of $I$, denoted by
$\ov I$, is the set of elements $x\in R$ such that $x$ satisfies an equation of
the 
form $$x^n+a_1x^{n-1}+\cdots +a_n=0$$ where $a_j\in I^j$ for $1\leq j\leq n$. If
$x\in \ov I$ we say $x$ is  integral over $I$. Note that $\ov I$ is an ideal. 
 An ideal $I$ is said to be {\bf integrally closed} or {\bf complete}  if $\ov
I=I$ and it is said to be {\bf normal} if all its powers are integrally closed.
An ideal $I$ is said to be 
{\bf asymptotically normal} if there exists an integer $N\geq 1$ such that $I^n$
is integrally closed for all $n\geq N$.
A ring $R$ is said to be {\bf analytically unramified} if its $\m$-adic
completion $\hat{R}$ is reduced. Rees in \cite{rees1}  showed that if $R$ is
analytically unramified then the integral closure filtration $\{\ov{I^n}\}$ is
an $I$-admissible filtration. It follows that if $I$ is an $\m$-primary ideal
then the {\bf normal Hilbert function} $\ov H_I(n)=\lm(R/\ov{I^n})$ coincides
with a polynomial $\ov P_I(n)$ of degree $d$ for large $n$. This polynomial is
referred to as the {\bf normal Hilbert polynomial} of $I$ and it is written in
the form
 $$\ov P_I(x)=\ov e_0(I){x+d-1\choose d}-\ov e_1(I){x+d-2\choose d-1}+\cdots
+(-1)^d\ov e_d(I)$$
 where $\ov e_0(I),\ldots ,\ov e_d(I)$ are integers uniquely determined by $I$.
These integers are known as {\bf normal Hilbert coefficients} of $I$. 
The integer $\ov e_1(I)$ is called the {\bf normal Chern number of $I.$} 
The Hilbert series of an $I$-admissible filtration $\mathcal{I}=\{I_n\}$, is
defined as
$$F_{\mathcal I}(t)=\sum_{n\geq 1}\lm(I_{n-1}/I_n)t^{n-1}.$$
Given any filtration $\mathcal I=\{I_n\},$  the Rees algebra, extended Rees
algebra and the associated graded ring of  $\mathcal I$ are defined as   
$$\mathcal R_+(\mathcal I)=\bigoplus_{n\geq 0}I_nt^n,\;\; 
\mathcal R(\mathcal I)=\bigoplus_{n\in\mathbb Z}I_nt^n\;\; 
\text{and}\;\; G(\mathcal
I)=\bigoplus_{n\geq 0}I_n/I_{n+1}$$
respectively.
 For the integral closure filtration $\mathcal I=\{\ov{I^n}\}$ we denote the
Hilbert series $F_{\mathcal I}(t) $ by $\ov F_I(t)$ the Rees algebra by
$\ov{\mathcal R_+ }(I)$, extended Rees algebra by $\ov{\mathcal R }(I)$ and the
associated graded ring
 by $\ov G(I)$. 

A {\bf reduction} of an $I$-admissible filtration $\mathcal I=\{I_n\}$ is an
ideal $J\subseteq I_1$ such that $JI_n=I_{n+1}$ for all large $n$. Equivalently
$J\subseteq I_1$ is a reduction 
of $\mathcal I$ if and only if $\mathcal R(\mathcal I)$ is a finite $\mathcal R
(J)$-module. A {\bf minimal reduction} of $\mathcal I$ is a reduction of
$\mathcal I$ minimal
with respect to inclusion. For a minimal reduction $J$ of $\mathcal I$, we set 
$$r_J(\mathcal I)= \sup \{n\in \mathbb Z \mid I_n\not= JI_{n-1}\}.$$ The 
{\bf reduction  number} $r(\mathcal I)$ of $\mathcal I$ is defined to be the
least $r_J(\mathcal I)$ over all possible minimal reductions $J$ of $\mathcal
I$. If $\mathcal I =\{\ov{I^n}\}$ then we write $r(\mathcal I):=\ov{r}(I).$

The coefficient $e_0(\mathcal I)$ has been studied extensively. We shall not
discuss it in this paper.
Our purpose here is to focus on the other coefficients
about which not much is known. 

In the first section, we discuss results about 
$e_1(\mathcal I)$ for an $I$-admissible filtration.  The principal result is a
theorem of Huckaba and Marley which
gives sharp lower and upper bounds for this coefficient and relates them to the
depth of $G(\mathcal I).$ Using this result, we discuss when $e_1(\mathcal I)$
vanishes. We will also discuss conditions under which  the  normal reduction
number is atmost one or atmost two. Use of Huckaba-Marley Theorem simplifies the
proofs of these results.

In the second section, we discuss a remarkable result due to
Moral\'es, Trung and Villamayor. It answers the question: when is $e_1(I)=\ov
e_1(I) $ for a parameter ideal in an analytically unramified excellent local
domain ? Their result states that it is possible only when $I$ is normal and $R$
is regular. We will see that this is true for  unmixed analytically unramified
local rings.

In the third  section we continue the study of the normal Chern number.  In this
section we sketch the recent solution of the Positivity Conjecture of
Vasconcelos due to 
Goto, Hong and Mandal \cite{ghm}. 

We shall survey main results about the second  Hilbert
coefficient in the fourth section. The most decisive result in this direction is
due to Marley. It states that in a Cohen-Macaulay local ring $e_2(\mathcal I)
\geq 0$ for any $I$-admissible filtration $\mathcal I.$ We discuss this for
$\ov e_2(I)$ using a formula of Blancafort for the difference of Hilbert
polynomial and Hilbert function of an $I$-admissible filtration in terms of
lengths of local cohomology modules of the Rees algebra $\mathcal R(\mathcal I)$
with support in $\bigoplus_{n>0} I_nt^n$. We shall give a new proof of  Itoh's
lower bound 
$\ov e_2(I) \geq  \ov e_1(I) -\lm(\ov I/I)$ and show that
equality holds if and only if $\ov r(I) \leq 2.$ And in this case $\ov G(I)$ is
Cohen-Macaulay.

In the fifth section, we study the third normal Hilbert coefficient.  This 
coefficient is perhaps the most interesting. Marley
constructed an example of a monomial ideal $I$ in the power series ring
$k[[x,y,z]]$ for which $e_3(I) =-1.$ In contrast
to this, Itoh proved that if $R$ is any analytically unramified Cohen-Macaulay
local ring then 
$\ov e_3(I) \geq 0.$ This is proved by using a local cohomological
interpretation of this coefficient. One has to invoke a special case of Itoh's
vanishing theorem : $[H^2_{\mathcal R(I)_+}\ov{\mathcal R}(I)]_n=0$ for all $ n
\leq 0$ 
if $\dim R \geq 3.$ Itoh proved that if $\ov r (I) \leq 2$ and $R$ is
Cohen-Macaulay analytically unramified local ring of dimension at least $3$ then
$\ov e_3(I)=0.$ 

In fact he showed that in this case $\ov G(I)$ is Cohen-Macaulay.  
Itoh conjectured that if $R$ is Gorenstein local ring of dimension at least
three and $\ov e_3(I)=0$ then $\ov r(I) \leq 2.$  We will present his solution
of this conjecture for the 
filtration $\{\ov{I^n}\}$
for parameter ideals $I$ for which $\ov I=\m.$  

The normal Hilbert polynomial of an $\m$-primary ideal 
in a two-dimensional regular local ring $ (R, \m)$ was computed 
by Rees and Lipman. 
In the sixth section, we present a proof using Zariski's theory of complete
ideals for computation of the normal
Hilbert function $\lm (R/ \ov{(I^rJ^s)})$ for $\m$-primary ideals $I$ and $J.$
The formula of Rees and Lipman is a special case of it. This formula also
follows from a formula
of Hoskin and Deligne for $\lm (R/I)$ where $I$ is a  complete $\m$-primary
ideal of $R.$ 

Finally in section seven, we study the normal Hilbert polynomial of a
zero-dimensional monomial ideal  in the polynomial ring $R=k[x_1,x_2, \ldots,
x_n]$ over a field 
$k.$ In this case, this  polynomial is expressed as a difference of two Ehrhart
polynomials 
derived from the exponent vectors of the monomials generating $I.$

\section{Study of $ e_1(\mathcal I)$}
In this section we discuss results about the first normal Hilbert coefficient 
$\ov e_1(I)$. The first general result was proved by  T. Marley in \cite{m}.
\begin{theorem}{\rm \cite[Lemma 3.19]{m}}\label{M}
 Let $(R,\m)$ be a  local ring and let $I$ be an $\m$-primary ideal. Let
$\mathcal I=\{I_n\}$ be an $I$-admissible filtration. Then $$e_1(\mathcal I)\geq
e_1(I).$$
In particular if $R$ is an analytically unramified Cohen-Macaulay local ring
then  $\ov e_1(I)\geq 0.$
\end{theorem}
\begin{proof}
We may assume the residue field $R/\m$ is infinite and then by using superficial
elements \cite[section 8.5] {hs} we can reduce to the case $\dm R = 1$. Since
$I^n\subseteq I_n$ for  all $n$, $P_{\mathcal I}(n)\leq P_I(n)$ for $n$
sufficiently large. Then
$$e_0(\mathcal I)n-e_1(\mathcal I)\leq e_0(I)n-e_1(I)$$
for $n$ sufficiently large. Since $e_0(\mathcal I)=e_0(I)$, so $e_1(\mathcal
I)\geq e_1(I)$.
Now let $R$ be an analytically unramified Cohen-Macaulay local ring. Let $J$ be
a minimal reduction of $I.$ Then
$\ov {I^n}=\ov {J^n}$ for all $n.$  Thus $$\ov e_1(I)=\ov e_1(J)\geq e_1(J)=0.$$
\end{proof}

Huckaba and Marley \cite{hm} have given lower and upper bounds for the 
Chern number of an $I$-admissible filtration in a Cohen-Macaulay local ring. 

\begin{theorem}[Huckaba-Marley Theorem] {\rm \cite[Theorem 4.7]{hm}}
\label{HM}
  Let $(R,\m)$ be a $d$-dimensional Cohen-Macaulay local ring and let $I$ be an
$\m$-primary ideal. Let $\mathcal I=\{I_n\}$ be an $I$-admissible filtration and
$J$ be a minimal reduction of $\mathcal I$. Then
 \begin{enumerate}
    \item[{\rm (1)}] $\displaystyle{\sum_{n\geq 1}\lm((I_n,J)/J)} \leq
e_1(\mathcal I)\leq \displaystyle{\sum_{n\geq 1}\lm(I_n/JI_{n-1})}$. 
    \item[{\rm (2)}]  $e_1(\mathcal I)=\displaystyle{\sum_{n\geq
1}\lm((I_n,J)/J)}  $ if and only if  $G(\mathcal I)$ is Cohen-Macaulay.
    \item [{\rm (3)}] $e_1(\mathcal I)= \displaystyle{\sum_{n\geq
1}\lm(I_n/JI_{n-1})}$ if and only if $\depth G(\mathcal I)\geq d-1$.
   \end{enumerate}
 \end{theorem} 
  
This theorem has been proved for modules in \cite[Theorem 2.5, 2.7]{rv}.
Before we describe consequences of the Huckaba-Marley Theorem,  we state the
Valabrega-Valla criterion for Cohen-Macaulayness of $G(\mathcal{ I})$ for an
$I$-admissible filtration $\mathcal I$. This plays a fundamental role in the
study of Hilbert polynomials.
 \begin{theorem} [Valabrega-Valla, \cite{vv}]\label{vv}
 Let $(R,\m)$ be a $d$-dimensional Cohen-Macaulay local ring and let $I$ be an
$\m$-primary ideal. Let $\mathcal{I}=\{ {I_n}\}$ be an $I$-admissible
filtration. Let 
 $J=(x_1,\ldots ,x_d)$ be a minimal reduction of $\mathcal{I}$ and $\underbar
x^*=x_1^*,\ldots ,x_d^*$ be their images in $I_1/I_2$. Then $\underbar x^*$  is
a regular
 sequence if and only if $J\cap I_n=JI_{n-1}$ for all $n\geq 1$.
 \end{theorem}

The next result of Marley shows the consequences of the vanishing of the Chern
number of an $I$-admissible filtration in Cohen-Macaulay local ring.
\begin{corollary}{\rm \cite[Theorem 3.21]{m}}
 Let $(R,\m)$ be a $d$-dimensional Cohen-Macaulay local ring and $I$ be an
$\m$-primary ideal. Let $\mathcal I=\{I_n\}$ be an $I$-admissible filtration.
Then the following are equivalent:
 \begin{enumerate}
  \item[{\rm (1)}] $I$ is generated by a system of parameters and $\mathcal
I=\{I^n\}$,
  \item[{\rm (2)}] $e_1(\mathcal I)=\cdots =e_d(\mathcal I)=0$,
  \item[{\rm (3)}] $e_1(\mathcal I)=0$.
 \end{enumerate}

\end{corollary}
\begin{proof}
 $ (1) \Rightarrow (2):$ Since $R$ is Cohen-Macaulay, $I$ is generated by a
regular sequence, $G(I)\simeq R/I[x_1,x_2, \ldots, x_d]$ where $x_1, x_2,
\ldots, x_d$ are indeterminates. Therefore for all $n \geq 1,$
$$P_I(n)=\lambda(R/I)\binom{n+d-1}{d}.$$
Hence $e_1(I)=e_2(I)=\cdots=e_d(I)=0.$\\
 $ (2) \Rightarrow (3):$ This is clear.\\
 $ (3) \Rightarrow (1):$ Let $e_1(\mathcal I)=0.$ Then
by Theorem \ref{HM}, $G(\mathcal I)$ is Cohen-Macaulay and
$I_n \subset J$ for all $n.$ By Theorem \ref {vv},
$I_n \cap J=JI_{n-1}=I_n$ for all $n \geq 1.$ This
gives $I_n=J^n$ for all $n \geq 1.$
\end{proof}

As a consequence of the above result  in an analytically unramified
Cohen-Macaulay local ring if $\ov e_1(I)=0$ then $\ov {I^n}=I^n$ for all $n,$
  $\ov e_j(I)=0$ for $j\geq 1$ and $\ov r(I)=0.$

  \begin{corollary}{\rm \cite[Corollary 4.9]{hm}}\label{s4}
  Let $(R,\m)$ be a $d$-dimensional Cohen-Macaulay local ring with $R/\m$
infinite and let $I$ be an $\m$-primary ideal. Let $\mathcal I=\{I_n\}$ be an
$I$-admissible filtration. Then
   $\lm(R/I_1)\geq e_0(\mathcal I)-e_1(\mathcal I)$ and equality holds if and
only if $r(\mathcal I)\leq 1$.
  
  \end{corollary}
\begin{proof} 
Let $J$ be a minimal reduction of $\mathcal I$. Since $$\lm((I_1,J)/J)\leq
\sum_{n\geq 1}\lm((I_n,J)/J) \leq e_1(\mathcal I)$$
we have 
$$\lm(I_1/J)=e_0(\mathcal I)-\lm(R/I_1) \leq e_1(\mathcal I).$$ 
If  $\lm(I_1/J)=e_0(\mathcal I)-\lm(R/I_1) = e_1(\mathcal I)$, we get
$\lm((I_n,J)/J)=0$ for all
$n \geq 2.$ By Theorem \ref{vv}, $I_n=JI_{n-1}$ for all
$n \geq 2.$ Hence $r(\mathcal I)\leq 1.$  Conversely let
$r(\mathcal I)\leq 1.$ Then by Theorem \ref{HM}, 
$\lm(I_1/J) \leq e_1(\mathcal I)\leq \lm(I_1/J). $ Hence
$\lm(I_1/J)=e_1(\mathcal I)=e_0(\mathcal I)-\lm(R/I_1).$
\end{proof}

 The next result of Huckaba-Marley characterizes  Cohen-Macaulay property  of
the Rees algebra of an $I$-admissible filtration in terms of its Chern number.
  \begin{corollary}{\rm \cite[Corollary 4.10]{hm}}
  Let $(R,\m)$ be a $d$-dimensional Cohen-Macaulay local ring and $I$ be an
$\m$-primary ideal. Let $R/\m$ be infinite. Let $\mathcal I=\{I_n\}$ be an
$I$-admissible filtration and $J$ be a minimal reduction of $\mathcal I$. Then
  $\mathcal R_+(\mathcal I)$ is Cohen-Macaulay if and only if $e_1(\mathcal
I)=\sum_{n=1}^{d-1}\lm((I_n,J)/J).$
 \end{corollary}
 
 \begin{proof}
 We first prove that  $e_1(\mathcal I)=\sum_{n=1}^{d-1}\lm((I_n,J)/J)$ if and
only if $G(\mathcal I)$ is Cohen-Macaulay and $r(\mathcal I)<d$.  The if part
follows from Theorem \ref{HM} (1). Now suppose the equality holds. By Theorem
\ref{HM}(1), $G(\mathcal I)$ is Cohen-Macaulay and $I_n\subseteq J$ for all
$n\geq d$. By Theorem \ref{vv}, $J\cap I_n=JI_{n-1}$ for all $n\geq 1$. Thus
$I_n=JI_{n-1}$ for ll $n\geq d$. The rest follows from \cite{viet}.
 \end{proof}
 The next result describes the relationship between the postulation number and
reduction number of an $I$-admissible filtration.
 \begin{theorem}{\rm\cite[Corollary 3.8]{m}}\label{redn}
 Let $(R,\m)$ be a $d$-dimensional Cohen-Macaulay local ring with an infinite
residue field. Let $I$ be an $\m$-primary ideal and $\mathcal I = \{I_n\}$ be an
$I$-admissible filtration such that $\depth G(\mathcal I)\geq d-1$. Then
$r(\mathcal I) = n(\mathcal I) + d.$
 \end{theorem}
 
 Itoh in \cite{i2} has given another  lower bound  for the normal Chern number
in a Cohen-Macaulay local ring. We give a different proof of Itoh's theorem.
Before we prove it we need some more results of Itoh proved in \cite{i3}.
 
 \begin{theorem}[Huneke-Itoh Intersection Theorem]\rm{\cite[Theorem
1]{i3}}\label{ithm2}
Let $(R,\m)$ be a $d$-dimensional local ring and $I$ be an ideal generated by a
regular sequence. Then for every $n\geq 1,$ 
$$I^n\cap \ov{I^{n+1}}=I^n\ov I.$$ 
 \end{theorem}

Recently Ulrich and Hong have given a simpler proof in a more 
general settings of the Huneke-Itoh Intersection Theorem \cite{hu}.
 
 \begin{lemma}\label{itl}
 Let $(R,\m)$ be a local ring and $I$ be a parameter ideal of $R$. Then $\ov
r(I)\leq 2$ if and only if $\ov {I^{n+2}}=I^n \ov {I^2}$ for all $n\geq 1$.
 \end{lemma}
 
 \begin{proof}
 We apply induction on $n$. If $\ov r(I)\leq 2$, $\ov {I^{n+2}}=I\ov {I^{n+1}}$
for all $n\geq 1$. Hence $\ov {I^3}=I\ov {I^2}$. Thus the result holds for
$n=1$. Now assume the results holds for $n$ and we will prove it for $n+1$.
  Note that $\ov {I^{n+3}}=I\ov {I^{n+2}}=I^{n+1} \ov{I^2}$. 
 Conversely assume that $\ov {I^{n+2}}=I^n \ov {I^2}$ for all $n\geq 1$. Apply
induction on $n$. Then $\ov {I^{n+3}}=I^{n+1}\ov {I^2}=I\ov {I^{n+2}}$. Hence
$\ov r(I)\leq 2$.
 \end{proof}

Next we prove a result of Itoh about a lower bound on the normal Chern number.
For a generalisation of this result to good filtration of modules, 
see \cite[Theorem 3.1]{rv}.
 \begin{theorem}{\rm \cite[Theorem 2(1)]{i2}}\label{s5}
  Let $(R,\m)$ be a $d$-dimensional analytically unramified Cohen-Macaulay local
ring and let $I$ be a parameter ideal. Then 
$$\ov e_1(I) \geq \lm(\ov I/I)+ 
 \lm(\ov {I^2}/I\ov I)$$ 
and equality holds if and only if $\ov r(I)\leq 2$.
 \end{theorem}
 
 \begin{proof}
 Applying Theorem \ref{HM},
 \begin{eqnarray*}
 \ov e_1(I)& \geq & \lm(\ov I/I)+\lm((\ov {I^2},I)/I)\\
                & = & \lm(\ov I/I)+\lm(\ov {I^2}/(\ov{I^2}\cap I))\\
                & = & \lm(\ov I/I)+\lm(\ov {I^2}/\ov{I}I) ~~~~~~(\mbox{by
Huneke-Itoh intersection theorem}).
 \end{eqnarray*}
 Suppose that equality holds, i.e. $\ov e_1(I)=\lm(\ov I/I)+\lm(\ov {I^2}/I\ov
I)$. By Theorem \ref{HM}, $\lm ((\ov {I^n},I)/I)=0$ for all $n\geq 3$, which
implies $\ov {I^n}\subseteq I$ for $n\geq 3$.  By Theorem \ref{HM}, $\ov G(I)$
is Cohen-Macaulay. Hence by Valabrega-Valla theorem we have $\ov {I^n}\cap
I=I\ov {I^{n-1}}$ for all $n\geq 1$. Thus we have $\ov {I^{n+2}}=\ov
{I^{n+2}}\cap I=I\ov {I^{n+1}}$ for all $n\geq 1$. Hence $\ov r(I)\leq 2$. \\
Conversely assume that  $\ov r(I)\leq 2$. Then we have $\ov {I^{n+2}}=I\ov
{I^{n+1}}$ for all $n\geq 1$. By Theorem \ref{HM}(1), we get $\ov e_1(I)\leq
\lm(\ov I/I)+\lm(\ov {I^2}/I\ov I)$. Hence we have $$\ov e_1(I)-\lm(\ov I/I)=
 \lm(\ov {I^2}/I\ov I).$$
 \end{proof}

 \section{On the equality $\ov e_1(I)=e_1(I)$ }
In this section we observe how the equality of the Chern number and normal Chern
number characterizes the ring to be regular. In this direction Moral{\'e}s,
Trung and Villamayor proved the following interesting 
\begin{theorem}{\rm \cite[Theorem 1,2]{mtv}}
 Let $(R,\m)$ be an analytically unramified excellent local domain and let $I$
be a parameter ideal. If $\ov e_1(I)=e_1(I)$ then $R$ is regular and $\ov
{I^n}=I^n$ for all $n$.
\end{theorem}
  We prove that this  result is true for analytically unramified  unmixed local
domains. A local ring $R$ is called {\bf unmixed} if $\dim \hat{R}/\mathcal P=\dim R$ 
for each associated prime ideal $P$ of $\hat{R}.$ 
  In order to prove the theorem we need to
  recall the theory of $\m$-full ideals introduced by Rees. First  we present
Goto's Theorem regarding the
   characterization of regular local rings in terms of $\m$-full ideals. We say
that an ideal $I$  is {\bf $\m$-full}, if $\m I: x = I$ for some $x \in \m.$

  \begin{lemma}{\rm \cite[Lemma 2.2]{goto}}\label{lgoto}
  Let $(R,\m)$ be a $d$-dimensional local ring and $I$ be an $\m$-primary ideal.
Then
  \begin{enumerate}
   \item[\rm {(1)}] $\mu(I)\leq \lm(\m I :x/\m I)=\lm(R/I+xR)+\mu(I+xR/xR)$ for
any element $x$ of $\m$.\\
   \item[{\rm (2)}] If $I$ is $\m$-full then $\mu (J)\leq \mu (I)$ for any ideal
$J$ of $R$ such that $I\subseteq J$. 
  \end{enumerate}
  \end{lemma}
  
  \begin{proof}
    {\rm (1)} Consider the exact sequence
   $$0\longrightarrow \m I:x/\m I\longrightarrow R/\m I\xrightarrow{x} R/\m
I\longrightarrow R/\m I+xR\longrightarrow 0.$$
   It follows from the above sequence that $\lm(\m I:x/\m I)=\lm(R/\m I+xR)$.
Since $I\subseteq \m I:x$, we have $\mu (I)=\lm(I/\m I)\leq \lm(\m I:x/\m I)$. 
  The equality follows from
\begin{eqnarray*} 
\lm(R/(I,x))+\mu((I,x)/(x))&=&\lm(R/(I,x))+\lm(I+xR/\m I+xR)\\&=&\lm(R/(\m
I,xR)).
\end{eqnarray*}
  {\rm (2)} Since $I$ is $\m$-full  there exists  $x\in \m$ such that $\m
I:x=I$. Hence $\mu(J)\leq \lm(J/\m I+xJ)$, as $\m I+xJ\subseteq \m J$. From the
exact sequence 
  given below
  $$0\longrightarrow I/\m I\longrightarrow J/\m I\xrightarrow{x} J/\m
I\longrightarrow J/\m I+xJ\longrightarrow 0$$
  it follows that $\mu(I)=\lm(J/\m I+xJ)$. Thus we have $\mu(J)\leq \mu(I)$.   
  \end{proof}
  
  \begin{proposition}{\rm \cite[Proposition 2.3]{goto}}\label{pgoto}
   Let $(R,\m)$ be a local ring of dimension $d\geq 1$ and $I$ be a parameter
ideal. Then
   $I$ is $\m$-full if and only if $R$ is regular and $\lm(I+\m^2/\m^2)\geq
d-1.$
  \end{proposition}
  
  \begin{proof}
   Let $I$ be $\m$-full. Let $x\in \m$ such that
    $\m I:x=I$. Then by Lemma \ref{lgoto}, we have $\mu(\m)\leq \mu(I)=d$. Thus
$R$ is regular. Note that by Lemma \ref{lgoto}, we have
    $$\lm(R/I+xR)+\mu(I+xR/xR)=\mu(I)=d.$$ But $\lm(R/I+xR)\geq 1$ and
$\mu(I+xR/xR)\geq d-1$. Hence we must have $I+xR=\m .$ Thus $\lm(I+\m
^2/\m^2)\geq d-1$.\\
  \indent Conversely assume that $R$ is regular and $\lm(I+\m^2/\m^2)\geq d-1$.
Choose $x\in \m \setminus \m^2$ such that
  $\m =I+xR$. Then we have $\lm(\m I:x/\m I)=d$ by Lemma \ref{lgoto}. We also
have $\mu(I)=\lm(I/\m I)=d$. Thus $\lm(I/\m I)=\lm( \m I:x/\m I)$ and hence
$I=\m I: x$. Therefore $I$ is $\m$-full.
  \end{proof}

  \begin{theorem}{\rm \cite[Theorem 2.4]{goto}}\label{tgoto}
   Let $(R,\m)$ be a $d$-dimensional local ring such that $R/\m$ is infinite. If
$\ov I=I$ then $I$ is $\m$-full or $I=\sqrt{ (0)}$.
  \end{theorem}

  \begin{theorem}{\rm \cite[Theorem 1]{mtv}}\label{thm1}
   Let $(R,\m)$ be a $d$-dimensional analytically unramified unmixed local ring
and $I$ be a parameter ideal. If $\ov e_1(I)=e_1(I)$ then 
  $R$ is a regular ring and  $\ov {I^n}=I^n$ for all $n.$
 
  \end{theorem}
  
  \begin{proof}

  Since $R$ is analytically unramified, $\lm(\ov{\m^n}/\ov{\m^{n+1}})$ is a
polynomial function of degree $d-1$ with leading coefficient $e(\m).$ Note that 
   $$\lm(\ov{\m^n}/\ov{\m^{n+1}})\leq \lm(\ov{\m^n}/\m
\ov{\m^n})=\mu(\ov{\m^n})\leq \mu(\ov{I^n}).$$
   The last inequality holds because $\ov{I^n}$ is $\m$-full by Theorem
\ref{tgoto}. We also have
   $$\mu(\ov{I^n})=\lm(\ov{I^n}/\m\ov{I^n})\leq \lm(\ov{I^n}/\m
I^n)=\lm(\ov{I^n}/I^n)+\mu(I^n).$$
   Since $\ov e_1(I)=e_1(I),$ \; $\lm(\ov{I^n}/I^n)$ is a polynomial function of
degree $<d-1$ while $\mu(I^n)={n+d-1\choose d-1}$. Thus we have $e(\m)=1$. Since
$R$ is 
  unmixed by Nagata's Theorem in \cite{nag}, $R$ is regular. 
 Without loss of generality we may assume that $R$ is a complete local ring. 

Let
$\ov{\mathcal R}(I)=\oplus_{n\geq 0}\ov{I^n}t^n$ and $\mathcal
R(I)=\oplus_{n\geq 0}I^n
  t^n$. Observe that $\ov{\mathcal R}(I)/\mathcal R(I)$ is a finite graded 
$\mathcal R(I)$-module. Since $\ov e_1(I)=e_1(I)$ the polynomial function
   $\lm(\ov{I^n}/I^n)$ is of degree $<d-1$. Suppose that there exists  
  $n$ such $\ov {I^n}\not=I^n$. As $\mathcal R(I)$ is 
  Cohen-Macaulay, it is the intersection of its localizations at the minimal
primes of principal ideals by \cite[Theorem 53]{k}. This shows that $\h \left[
\mathcal R(I):_{\mathcal R(I)}\ov{\mathcal R}(I)\right] =1$. Thus $\dm
\ov{\mathcal R}(I)/
  \mathcal R(I)=d$. Hence $\lm(\ov{I^n}/I^n)$ is a polynomial function of degree
$d-1$, which is a contradiction. Therefore $\ov{I^n}=I^n$ for all $n.$   
  \end{proof}

\section{The Positivity Conjecture}
At a conference held in 2008 in Yokohama, Japan, Wolmer Vasconcelos \cite{vas}
announced several conjectures about the the Chern number of a parameter ideal
and the normal Chern number of an $\m$-primary ideal in a Noetherian local ring
$(R,\m).$  First we quote his conjecture for the normal Chern number and then
sketch a solution of the conjecture.
\begin{conjecture}{\rm \cite[Vasconcelos]{vas}}
 Let $(R,\m)$ be a $d$-dimensional analytically unramified local ring and let
$I$ be an $\m$-primary ideal. Then $\ov e_1(I)\geq 0$.
\end{conjecture}
 We show that the Positivity Conjecture holds for the filtration
$\overline{\m^n}$ where $\m$ is the maximal homogeneous ideal of the face ring
of a pure simplicial
  complex $\Delta.$ Let $\Delta$ be a $(d-1)$-dimensional simplicial complex.
Let $f_i$ denote the number of $i$-dimensional faces of
  $\Delta$ for $i=-1,0, \ldots, d-1.$ Here $f_{-1}=1.$ Let $\Delta$ have $n$
vertices $\{v_1, v_2, \ldots, v_n\}.$ Let $x_1,x_2, \ldots,x_n$ be
indeterminates over a field $k.$ The ideal $I_{\Delta}$ of $\Delta$ is the ideal
generated by the square free monomials
  $x_{a_1}x_{a_2}\ldots x_{a_m}$ where $1 \leq a_1 < a_2 < \cdots < a_m \leq n$
and $\{v_{a_1},v_{a_2}, \ldots,v_{a_m}\} \notin \Delta.$   The face ring of
$\Delta$ over a field $k$ is defined as 
$k[\Delta]=k[x_1,x_2,\ldots,x_n]/I_{\Delta}.$
  
  \begin{lemma}\label{red}
   Let $R$ be a Noetherian ring and $I$ be an ideal of $R$ such that the
associated graded ring $G(I)=\bigoplus_{n=0}^\infty I^n/I^{n+1}$ is reduced.
Then $\ov{I^n}=I^n$ for all $n$.
  \end{lemma}
  
  \begin{proof}
   Let $\mathcal R(I)=\bigoplus_{n \in \Z} I^nt^n$ denote the extended Rees ring
of $I.$  Since $G(I) \simeq \mathcal R(I)/(u)$, where $u=t^{-1},$  is reduced,
$(u)=P_1\cap P_2 \cap \ldots \cap P_r$ for some height one prime ideals $P_1,
\ldots, P_r$ of $\mathcal R(I).$ Therefore
  $(u)$ is integrally closed in $\mathcal R(I).$ As  $P_i\mathcal
R(I)_{P_i}=(u)\mathcal R(I)_{P_i}$ for all $i,$ 
  $\mathcal R(I)_{P_i}$ is a DVR for all $i.$ Since $u$ is  
  regular, $\text{Ass}(\mathcal R(I)/(u^n))=\{P_1, P_2, \ldots, P_r\}$
  for all  $n\geq 1.$ Thus $(u^n)=\cap_{i=1}^rP_i^{(n)}$ is integrally closed.
Hence  ${I^n}=(u^n)\cap R$  is integrally closed for all $n.$
  \end{proof}

   \begin{theorem}\label{chernface}
    Let $\Delta$ be a simplicial complex of dimension $d-1$. Let $\m$ denote the
maximal homogeneous ideal of the face ring $k[\Delta]$ over a field $k$. Then
    \begin{enumerate}
 \item [{\rm  (1)}] $\ov {\m^n}=\m^n$ for all $n$.
 \item[{\rm  (2)}] $\ov e_1(\m)=e_1(\m)=df_{d-1}-f_{d-2}$.
 \item[ {\rm (3)}] If $\Delta $ is pure then $\ov{e}_1(\m)=e_1(\m)\geq 0.$
    \end{enumerate}
   \end{theorem}
   
   \begin{proof}
   {\rm (1)} Since $k[\Delta]$ is a standard graded $k$-algebra,
$G(\m)=k[\Delta].$ Hence $G(\m)$ is reduced and then by Lemma \ref{red},
$\ov{\m^n}=\m^n$ for all $n\geq 0$. \\
   {\rm (2)} Since  $\lm(\m^n/\m^{n+1})=\dim_kk[\Delta]_n,$  
   The Hilbert Series of the face ring is  
   $$H(k[\Delta],t)=\sum_{n=0}^\infty \dim_kk[\Delta]_n
t^n=\frac{h_0+h_1t+\cdots+h_st^s}{(1-t)^d}.$$ Put $h(t)=h_0+h_1t+\cdots +
h_st^s$. The face vector $(f_{-1}, f_1, f_0, \ldots,f_{d-1})$ and the $h$-vector
are related by the equation
   $$\sum_{i=0}^sh_it^i=\sum_{i=0}^df_{i-1}t^i(1-t)^{(d-i)}$$
   by \cite[Lemma 5.1.8]{bh}.
   Then by \cite[Proposition 4.1.9]{bh} we have 
   $$e_1(\m)=\ov e_1(\m)=h'(1)=df_{d-1}-f_{d-2}.$$ 
   {\rm (3)} Now we prove that if $\Delta$ is a pure simplicial complex then
$df_{d-1}\geq f_{d-2}$. Let $\sigma$ be a facet. For any
$v_i\in\sigma=\{v_1,\ldots , v_d\}$, $\sigma\setminus \{v_i\}$ is a
$(d-2)$-dimensional face and $\sigma\setminus \{v_i\}$ are distinct for all
$i=1,\ldots ,d$. Therefore each facet gives rise to $d$ faces of dimension 
$d-2.$  But different facets may produce same faces of dimension $d-2$. Since
$\Delta$ is pure,  each $(d-2)$-dimensional face is contained in a facet. Hence
$df_{d-1}\geq f_{d-2}$. Therefore
   $\ov{e}_1(\m)\geq 0$ by Theorem \ref{chernface}.
 \end{proof}
 The above theorem indicates that the maximal homogeneous ideal of the  face
ring of a non-pure simplicial complex may have negative Chern number. Indeed,
consider the simplicial complex 
   $\Delta_n$ with its vertex set as  $\{v_1,v_2, \ldots, v_{n+2}\}$ where $n
\geq 2$  and $$\Delta_n=\{\{v_1,v_2\},v_3, \ldots,v_{n+2} \}.$$
   Then $e_1(\m)=df_{d-1}-f_{d-2}=-n.$
   Hence we need to add the assumption of unmixedness on the ring in
Vasconcelos' Positivity Conjecture.

By a  {\bf finite cover} $S/R,$ we mean a ring extension $R\subseteq S$  such
that $S$ is a finite $R$-module. Then $S$ is a Noetherian semilocal ring.  We
say that the finite cover $S/R$ is {\bf birational} if $R$ is reduced and $S$ is
contained in the total ring of fractions of $R;$ that $S/R$ is of finite length
if $\lm (S/R)$ is finite;  and that $S/R$ is  Cohen-Macaulay if $S$ is
Cohen-Macaulay as an $R$-module.   
 
   \begin{theorem}[Mandal-Singh-Verma, \cite{msv}]  \label{cor-pc}  Let $(R,\m)$
be an analytically unramified Noetherian local ring of positive dimension.  Then
$\ov e_1(I)\geq 0$ for all $\m$-primary ideals $I$ of $R$ if $R$ satisfies any
one of the following conditions: 
   \begin{enumerate}
   \item[{\rm  (1)}] $R$ has a finite Cohen-Macaulay cover which is of finite
length or is birational,
   \item[{\rm (2)}] $\dim R=1$, 
   
  \item[{\rm (3)}] The integral closure of $R$ is Cohen-Macaulay as  an
$R$-module,
   
   \item[{\rm (4)}] $\dim R=2$ and all maximal ideals of the integral closure of
$R$ have the same height,
   
   \item[{\rm (5)}] $R$ is a complete local integral domain of dimension 2.     
    
   \end{enumerate}
   \end{theorem} 
   
   In \cite{ghm}  the solution to the Positivity Conjecture has been given for
analytically unramified unmixed local rings. In order to prove the Theorem we
need a lemma. First we set up the notation for the lemma. Suppose $z_1, z_2,
\ldots, z_d$ are indeterminates. Let $R'=R[z_1,\ldots , z_d]_\mathfrak q$ where
$\mathfrak q={\m R[z_1,\ldots , z_d]}$ and take quotient by a general element
$x=a_1z_1+\cdots +a_dz_d$, where $I=(a_1,\ldots ,a_d)$ and consider the ring
$D'=R'/(x)$. For a function $f : \mathbb Z \rightarrow \mathbb Z$ we define
$\Delta f(n)=f(n)-f(n-1).$ The next lemma is crucial for applying induction
on the dimension of the ring. It is proved  in \cite {i2} and \cite{ghm}.

   \begin{lemma}\label{hu}
   Let $(R,\m)$ be a complete normal local domain. Then with the notation above 
   \begin{enumerate}
   \item[{\rm (1)}] $\ov{I^n}D'=\ov{I^n D'}$ for large $n$.
   \item[{\rm (2)}] $\ov{P}_{ID'}(n)=\Delta \ov P_{IR'}(n).$
   \end{enumerate}
   \end{lemma}
   \begin{theorem}[Goto-Hong-Mandal, \cite{ghm}]\label{pos}
   Let $(R,\m)$ be an analytically unramified unmixed local ring of dimension $d
\geq 2.$ Then for every $\m$-primary ideal $I,$
   $$\ov e_1(I)\geq 0.$$
   \end{theorem}
   
   \begin{proof}
   We  sketch the proof. Without loss of generality we may assume $R$ is
complete and $I$ is a parameter ideal. We prove the theorem by induction on $d.$
 Let $S = \ov{R}$. For each $\p \in \Ass R$ we put  $S(\p)=\ov{R/\p}$. Then
$S(\p)$ is a module-finite extension of $R/\p$ and we get by  \cite[Corollary
2.1.13]{hs}
   $$S=\prod_{\p \in \Ass R}S(\p) \ \ ~~~\text{and}~~~\ \  \ov{I^{n}} =
\ov{I^{n}S}\cap R$$
   for all $n \ge 1$. Consider the map $\chi :R/\ov{I^n}\longrightarrow
S/\ov{I^nS}$. Then $\ker \chi =\ov{I^nS}\cap R=\ov{I^n}$. So $\chi$ is
injective.
   Hence
   \begin{eqnarray*}
    \lm_R(R/\ov{I^{n}}) \leq \lm_{R}(S/\ov{I^{n}S}) &=&  \sum_{\p \in \Ass R}
\lm_{R}(S(\p)/\ov{I^{n}S(\p)})\\
   &=& \sum_{\p \in \Ass
R}\lambda_{R}(S(\p)/\m_{S(\p)})\lambda_{S(\p)}(S(\p)/\ov{I^{n}S(\p)}),
   \end{eqnarray*}
   where $\m_{S(\p)}$ denotes the maximal ideal of $S(\p)$. 
   As $\dim S(\p) = d$ for each $\p \in \Ass R$, we have 
   \begin{eqnarray*}
   \overline e_0(I,R) = e_0(I,R) = e_0(I,S) &=&\sum_{\p \in \Ass
R}e_0(I,S(\p))\\
   & = & \sum_{\p \in \Ass R}\lambda_R(S(\p)/\m_{S(\p)})e_0({IS(\p)},S(\p))\\
   &=& \sum_{\p \in \Ass R}\lambda_R(S(\p)/\m_{S(\p)})\overline
e_0({IS(\p)},S(\p)).
   \end{eqnarray*}
   Therefore
   \begin{eqnarray*}
   0 &\le&\lambda_R(S/\ov{I^{n+1}S})- \lambda_R(R/\ov{I^{n+1}})
\\&=&\left[\overline{e}_1(I,R) - \sum_{\p \in 
\Ass R}\lambda_R(S(\p)/\m_{S(\p)})\overline{e}_1({IS(\p)},S(\p))
\right]\binom{n+d-1}{d-1}\\
 & + & \mbox{ terms of lower degree in}\;  n. 
   \end{eqnarray*}
   Hence
   \[ \overline{e}_1({I},R) \geq  \sum_{\p \in \Ass R}
\lambda_R(S(\p)/\m_{S(\p)})\ov{e}_1({IS(\p)},S(\p)). \] 
   In order to prove $\ov{e}_1(I,R) \ge 0$, it suffices to show that
$\ov{e}_1({IS(\p)},S(\p)) \geq 0$ for each $\p \in \Ass R$. If $d=2$, as $S(\p)$
is a Cohen-Macaulay local ring, $\ov{e}_1(I,R) \ge 0$. 
    Suppose that $d \geq 3$ and that our assertion holds true for $d-1$. 
Passing to the ring $S(\p)$, we may assume that $R$ is a complete  local normal
domain. Then we consider the general extension ring $T=R[z_1, z_2, \ldots, z_d]$
and $R'=R[z_1,\ldots , z_d]_\mathfrak q$ where $\mathfrak q={\m R[z_1,\ldots,
z_d]}$ and take quotient by a general element $x=a_1z_1+\cdots +a_dz_d$, where
$I=(a_1,\ldots ,a_d)$ and consider the ring $D'=R'/(x)$.  Then $D'$ is unmixed
and analytically unramified. 
By using Lemma \ref{hu} we have
$\ov{I^n}D'=\ov{I^n D'}$ for large $n$ and $$\ov{P}_{ID'}(n)=\Delta \ov
P_{IR'}(n).$$
   Then using induction on $d$ we are done.
  \end{proof}

\section{Study of $\ov e_2(I)$}

\noindent In this section we survey results on $\ov{e}_2(I)$. As in case of
$\ov{e}_1(I),$ it turns out that $\ov{e}_2(I)$ is also 
non-negative in analytically unramified Cohen-Macaulay local rings. Itoh has
given a lower bound for $\ov{e}_2(I)$ and has given a necessary 
and sufficient condition for the bound to be attained. We  give a proof of this
theorem. As a consequence we prove Huneke's 
theorem which gives a necessary and sufficient condition for vanishing of
$\ov{e}_2(I)$ in terms of reduction number $\ov r(I).$ We 
give a necessary and sufficient condition for $\ov{e}_2(I)=1$. We also discuss a
necessary and sufficient condition for $\ov r(I) \leq 2$ 
in terms of $\ov{e}_2(I).$ As a consequence we derive a theorem of Corso, Polini
and Rossi {\rm \cite[Theorem 3.12]{cpr}} which 
gives a condition for a normal ideal to have reduction number two. \\
\noindent T. Marley in his thesis {\rm \cite[Proposition 3.23]{m}} proved that
 $e_2(\mathcal{I})\geq 0$ for any $I$-admissible filtration $\mathcal{I}$ in a
Cohen-Macaulay local ring. 
\begin{theorem}{\rm \cite[Proposition 3.23]{m}}
 Let $(R,\m)$ be a $d$-dimensional Cohen-Macaulay local ring and $I$ be an
$\m$-primary ideal. Let $\mathcal I=\{I_n\}$ be an 
$I$-admissible filtration. Then $e_2(\mathcal I)\geq 0$.
\end{theorem}
\noindent
We give a proof of positivity of $\ov{e}_2(I)$. In order to prove this we will
use the following results.
\begin{theorem}{\rm \cite[Theorem 4.1]{b}}\label{blanc}
 Let $(R,\m)$ be a $d$-dimensional local ring and $I$ be an $\m$-primary ideal.
Let $\mathcal I=\{I_n\}$ be an $I$-admissible 
filtration.  Let $\mathcal{R}_+=\bigoplus_{n>0}I^nt^n$.
Then\\
{\rm (1)} For all $i\geq 0$, $\lm([H^i_{\mathcal R_+}(\mathcal
R(\mathcal{I}))]_n)<\infty $ for all $n\in \mathbb Z$. \\
{\rm (2)} For all $n\in \mathbb Z$,
$$P_{\mathcal I}(n)-H_{\mathcal I}(n)=\sum_{i=0}^d(-1)^i\lm([H^i_{\mathcal
R_+}(\mathcal R(\mathcal{I}))]_n).$$

\end{theorem}

\noindent
The following result  of Itoh is a key point in proving results about normal
Hilbert polynomials using  induction on the dimension of $R.$ See also
\cite {ci}.
\begin{theorem}{\rm \cite[Theorem 1 and Corollary 8]{i2}}\label{s1}
Let $I$ a parameter ideal of a Cohen-Macaulay analytically
unramified local ring $(R, \m)$ of dimension $d.$ Then there exists a system of
generators $x_1,\ldots,x_d$ of $I$ such that, if we put $C=R(T)/(\sum_i x_iT_i)$
and $J=IC$, where 
$R(T)=R[T]_{\m[T]}$ and $T=(T_1,\ldots,T_d)$ is a set of $d$ indeterminates. 
Then \\
{\rm (1)} $\ov{J^n} \cap R=\ov{I^n}$ for every $n \geq 0,$ \\
{\rm (2)} $\ov{J}=\ov{I}C,$ \\
{\rm (3)} $\ov{J^n}=\ov{I^n}C \cong \ov{I^n}R(T)/(\sum_i
x_iT_i)\ov{I^{n-1}}R(T)$ for large $n,$ \\
{\rm (4)} If $R$ is analytically normal and $\dim R \geq 3,$ then $C$ is
normal.\\
{\rm (5)} $\ov{e}_i(I)=\ov{e}_i(J)$ for $i=0,1,\ldots,d-1.$
\end{theorem}

We quote some results of Itoh about vanishing of graded components of local
cohomology modules. See also a recent paper of 
Hong and Ulrich \cite{hu}.
\begin{theorem}{\rm \cite[Proposition 13]{i2}}\label{s2} Let $(R, \m)$ be an
analytically unramified  Cohen-Macaulay local ring of
 dimension $d \geq 2.$ Let $N=It\mathcal{R}(I)$ and 
$M=(t^{-1}, I t)\mathcal{R}(I).$
 Then \\
{\rm (1) } $H^0_M(\ov{\mathcal{R}}(I))=H^1_M(\ov{\mathcal{R}}(I))=0,$ \\
{\rm  (2) } $[H^2_M(\ov{\mathcal{R}}(I)]_j=0$ for $j \leq 0,$ \\
{\rm (3)} $H^i_M(\ov{\mathcal{R}}(I))=H^i_N(\ov{\mathcal{R}}(I))$ for
$i=0,1,\ldots,d-1.$

\end{theorem}
 \noindent Now we prove  nonnegativity  of $\ov{e}_2(I).$ 
For a generalisation to good filtrations in Cohen-Macaulay modules, see
\cite[Proposition 3.1]{rv}.

\begin{theorem}
 Let $(R,\m)$ be an analytically unramified Cohen-Macaulay local ring of
dimension $d \geq 2$ and let $I$ be an $\m$-primary ideal. Then $$\ov{e}_2(I)
\geq 0.$$
\end{theorem}
\begin{proof}
We apply induction on $d$. Let $d=2$. By Theorem \ref{blanc},   
$$ \ov{P}_I(n)-\ov{H}_I(n)=\sum_{i=0}^2
(-1)^i\lm([H^i_N(\ov{\mathcal{R}}(I))]_n)$$ 
for all $n \in \mathbb{Z}$. Taking $n=0$ we get $
\ov{e}_2(I)=\lm([H^2_N(\ov{\mathcal{R}}(I))]_0) \geq 0.$ Let $d > 2$. Let $C$
and $J$ 
be as in Theorem \ref{s1}. Then by induction hypothesis $\ov{e}_2(I)=\ov{e}_2(J)
\geq 0$. 
\end{proof}

\noindent
The following theorem of Itoh gives a lower bound on $\ov{e}_2(I).$ For a generalisation
for good filtration of modules see \cite[Theorem 3.1]{rv}.

\begin{theorem}{\rm \cite[Theorem 2(2)]{i2}}\label{s3}
Let $(R,\m)$ be a $d$-dimensional analytically unramified Cohen-Macaulay local
ring.
 Let $I$ be a parameter ideal. Then $$\ov e_2(I)\geq \ov e_1(I)-\lm(\ov I/I).$$ 
\end{theorem}

\begin{proof}
Apply induction on $d$. For $d=2$, by Theorems \ref{blanc} and \ref{s2} ,
$$\ov{P}_I(1)-\ov{H}_I(1)=\lm([H^2_N(\ov{R}(I))]_1) \geq 0.$$ 
Hence $\ov{e}_0(I)-\ov{e}_1(I)+\ov{e}_2(I) \geq \lambda(R/\ov{I})$. Since $I$ is
a 
parameter ideal, $\ov{e}_0(I)=\lambda(R/I)$. Therefore $\ov{e}_2(I) \geq
\ov{e}_1(I)-\lambda(\ov{I}/I)$. Let $d >2 $. Let $C$ and 
$J$ be as in Theorem \ref{s1}. Since $\ov{J}=\ov{I}C$,
$\lm(\ov{I}/I)=\lm(\ov{J}/J)$. Hence the inequality follows by induction 
hypothesis. 
\end{proof}

\begin{theorem}
Under the assumptions of Theorem \ref{s3},
\begin{enumerate}
  \item[{\rm (1)}]  {\rm \cite[Theorem 2(2)]{i2}}\label{s7}
 $\ov{e}_2(I)=\ov{e}_1(I)-\lm(\ov{I}/I) \;\; \text{if and only if}\;\;  
\ov r(I) \leq 2.$
 \item[{\rm (2)}]  {\rm \cite[Theorem 3.12]{cpr}} If $\ov r(I)\leq 2$ then $\ov
G(I)$ is Cohen-Macaulay and its Hilbert series is given by
 $$\ov F_I(t)=\dfrac{\lm(R/\ov I)+[\ov e_0(I)-\lm(R/\ov I)-\lm(\ov{I^2}/I\ov
I)]t+\lm(\ov{I^2}/I\ov I)t^2}{(1-t)^d}.$$
\end{enumerate}

\end{theorem}
\begin{proof} Let $\ov{r}(I) \leq 2$. By Huckaba-Marley Theorem,  
$$ \sum_{n\geq1}\lambda(\ov{I^n}/I \cap \ov{I^{n}}) \leq \ov{e}_1(I) \leq \sum
_{n\geq 1} \lambda(\ov{I^n}/I\ov{I^{n-1}})=
\lm(\ov{I}/I)+\lm(\ov{I^2}/I\ov{I}).$$ Using Huneke-Itoh Intersection theorem we
get $I \cap \ov{I^2}=I\ov{I}$. Therefore 
$$\ov{e}_1(I)=\sum_{n \geq 1}\lambda(\ov{I^n}/I \cap \ov{I^{n}})=\lm(\ov
I/I)+\lm(\ov{I^2}/I\ov I).$$
Hence by Theorem \ref{HM}, $\ov{G}(I)$ is Cohen-Macaulay. Let $I=(x_1, x_2,
\ldots, x_d).$ Let $x_i^*$ denote the image of $x_i$ in 
$\ov I/\ov {I^2}.$ Let $\mathcal{F}=\{\ov{I^n} +I/I\} $. Then 
\begin{eqnarray*}
 \ov{G}(I)/(x_1^*,\cdots,x_d^*)&=& G(\mathcal{F})\\
&=& \bigoplus_{n \geq 0} \frac{\ov{I^n}+I}{\ov{I^{n+1}}+I}\\
&=& \frac{R}{\ov{I}} \oplus \frac{\ov{I}}{\ov{I^2}+I} \oplus
\frac{\ov{I^2}+I}{\ov{I^3}+I}.
\end{eqnarray*}
Therefore 
 $$H(\ov{G}(I),t) = \frac{\lambda(R/\ov I)+\lambda(\ov{I}/\ov{I^2}+I)t+
 \lambda(\ov{I^2}+I/I)t^2}{(1-t)^d}.$$
Hence 
$$\ov{e}_2(I) = \lambda(\ov{I^2}+I/I) = \lambda(\ov{I^2}/ I \cap
\ov{I^2})=\lambda(\ov{I^2}/I\ov{I}).$$ Also 
$\ov{e}_1(I)=\lambda(\ov{I}/I)+\lambda(\ov{I^2}/I \ov{I})$. Therefore
$\ov{e}_2(I)=\ov{e}_1(I)-\lambda(\ov{I}/I)$.\\
\noindent Conversely let $\ov{e}_2(I)=\ov{e}_1(I)-\lm(\ov{I}/I)$. Use induction
on $d$. Let $d=2$. Since $\lm(R/\ov{I})=
\ov{e}_0(I)-\ov{e}_1(I)+\ov{e}_2(I)$, $\ov{P}_I(1)=\ov{H}_I(1)$. Therefore by
Theorems \ref{blanc} and \ref{s2}, 
$\lm([H^2_N(\ov{\mathcal{R}}(I))]_1)=0$. By \cite[Lemma 4.3.5]{b1},
$\lm([H^2_N(\ov{\mathcal{R}}(I))]_n)=0$ for all $n \geq 1$. 
Hence $\ov{P}_I(n)=\ov{H}_I(n)$ for 
all $n \geq 1.$   
Hence $\ov{n}(I) \leq 0$. By Theorem \ref{redn}, $\ov{r}(I) \leq 2$. 
Consider $C$ as and $J$ as in theorem \ref{s1}. Then by induction
$\ov{J^{n+2}}=J^n\ov{J^2}$ for all $n \geq 0$. Therefore by  \cite[Proposition
17]{i2} $\ov{I^{n+2}}=I^n\ov{I^2}$ for  all $n \geq 0.$ Hence $\ov r(I) \leq
2.$ 
\end{proof}

\noindent Now we analyse vanishing of $\ov{e}_2(I)$. By Theorem \ref{s4},
$\lm(R/I_1)=e_0(\mathcal{F})-e_1(\mathcal{F})$ 
if and only if $r(\mathcal{F})\leq 1$. In this case, $G(\mathcal{F})$ is
Cohen-Macaulay by Huckaba-Marley Theorem. Hence 
$e_2(\mathcal{F})=0$. For integral closure filtration,
$\mathcal{F}=\{\ov{I^n}\}$, converse is also true. In other words 
vanishing of $\ov{e}_2(I)$ is sufficient to gurantee that $\ov{r}(I) \leq 1$.
Huneke \cite[Theorem 4.5]{hun1}  proved this 
if $d=2$. 

 \begin{theorem}\label{s8}
  Let $(R,\m)$ be a $d$-dimensional analytically unramified Cohen-Macaulay local
ring with infinite residue field and let $I$ be 
an $\m$-primary ideal. Then $\ov e_2(I)=0$ if and only if $\ov{r}(I) \leq 1 .$ 
 \end{theorem}
\begin{proof}
Let $\ov{e}_2(I)=0$. Let $J$ be a minimal reduction of $I.$  
By Theorem  \ref{s3} and Theorem \ref{s5}, 
$$0=\ov{e}_2(J) \geq \ov{e}_1(J)-\lambda(\ov{J}/J) \geq
\lambda(\ov{J^2}/J\ov{J}).$$ 
Therefore $\lm(\ov{J^2}/J\ov{J})=0=\ov{e}_1(J)-\lm(\ov{J}/J)$. Hence
$\ov{e}_1(I)=\lm(\ov{I}/J)$. By Theorems \ref{s4}, $\ov{r}
(I) \leq 1$. Conversely let $\ov{r}(I) \leq 1$. We may assume that $I$ is a
parameter ideal. Then by Theorem \ref{s7},
$\ov{e}_2(I)=\ov{e}_1(I)-\lm(\ov{I}/I)$. By Theorem 
\ref{s4}, $\ov{e}_1(I)-\lm(\ov{I}/I)=0$. So $\ov{e}_2(I)=0$.

\end{proof}
\noindent Next proposition gives a necessary and sufficient condition for
$\ov{e}_2(I)=1$.
\begin{proposition}\cite[Theorem 9]{i1}
 Let $(R,\m)$ be an analytically unramified Cohen Macaulay local ring of
dimension $d$. Let $I$ be a parameter ideal. Then
$\ov{e}_1(I)=\ov{e}_0(I)-\lm(R/\ov{I})+1$ if and 
only if $\ov{e}_2(I)=1.$
\noindent In this case, $\ov{r}(I) = 2$ and 
$$\ov F_I(t)=\dfrac{\lm(R/\ov I)+[\ov e_0(I)-\lm(R/\ov I)+1]t+t^2}{(1-t)^d}.$$
\end{proposition}
\begin{proof} Let $\ov{e}_1(I)=\ov{e}_0(I)-\lm(R/\ov{I})+1$. We use induction on
$d$ to prove that $\ov{e}_2(I)=1$. Let $d=2$. 
Let $J \subseteq I$ be a minimal reduction of $I$. Then by Huckaba-Marley
Theorem,
$$ \sum_{n \geq 2}\lm(\ov{J^n}/J \cap \ov{J^n}) \leq
\ov{e}_1(J)-\lm(\ov{J}/J).$$
Therefore $ \sum_{n \geq 2} \lm(\ov{J^n}/J \cap \ov{J^n}) \leq 1$. Suppose
$\sum_{n\geq 2} \lm(\ov{J^n}/J \cap \ov{J^n})=0$. 
Then $\ov{J^2}=J \cap \ov{J^2}=J\ov{J}$, by Huneke-Itoh intersection theorem. By
\cite[Theorem 3.26]{m} $\ov{e}_2(J)=0$. 
But Theorem \ref{s3} implies that $$\ov{e}_2(J)\geq
\ov{e}_1(J)-\lm(\ov{J}/J)=\ov{e}_1(I)-\ov{e}_0(I)+\lm(R/\ov{I})=1,$$ a
contradiction. 
Hence $\sum_{n \geq 2}\lm(\ov{J^n}/J \cap \ov{J^n})=1$. Therefore by Theorem
\ref{HM}, $\ov{G}(I)$ 
is Cohen-Macaulay. Similar argument as above shows that $\ov{J^2} \neq J\ov{J}$.
Therefore $\ov{J^n}=J \cap \ov{J^n}$ for all $n \geq 
3$. Hence by Valabrega-Valla theorem, $\ov{J^n}=J \ov{J^{n-1}}$ for all $n \geq
3$. Therefore Theorem \ref{s7} implies that 
$\ov{e}_2(J)=\ov{e}_1(J)-\lm(\ov{J}/J)=1$. Hence $\ov{e}_2(I)=1$. Let $d >2$. We
may assume that $I$ is a parameter ideal. Let $C$ and $J$ be as in theorem
\ref{s1}. Hence 
$\ov{e}_1(J)=\ov{e}_0(J)-\lm(C/\ov{J})+1$. Therefore by induction,
$\ov{e}_2(J)=1$. Hence $\ov{e}_2(I)=1$.

Conversely, let $\ov{e}_2(I)=1$. Let $J$ be a minimal reduction of $I$. 
Then by Theorem \ref{s3} and \ref{s5}, $$1=\ov{e}_2(J)\geq
\ov{e}_1(J)-\lm(\ov{J}/J) \geq 0.$$ Suppose 
$\ov{e}_1(J)-\lm(\ov{J}/J)=0$. Then Huckaba-Marley Theorem implies that
$\ov{e}_1(J)=\sum_{n\geq 1} \lm(\ov{J^n}/J \cap \ov{J^n})$ and $\ov{G}(I)$ is
Cohen-Macaulay. Therefore 
$\ov{J^n}=J \cap \ov{J^n}$ for all $n \geq 2$. By Valabrega-Valla theorem,
$\ov{J^n}=J \ov{J^{n-1}}$ for all $n\geq 2$. Hence 
$\ov{r}(J) \leq 1$. By Theorem \ref{s8}, $\ov{e}_2(J)=0$, a contradiction.
Therefore $\ov{e}_1(J)-\lm(\ov{J}/J)=1$ and hence
$\ov{e}_1(I)=\ov{e}_0(I)-\lm(R/\ov{I})+1$.

\end{proof}

\noindent As a consequence of Theorem \ref{s7} we get similar result for normal
ideals in \cite[Theorem 3.12]{cpr}.
However the following example in \cite{cpr} shows that 
$\lm(R/I)=e_0(I)-e_1(I)+e_2(I)$ is not a sufficient condition to guarantee that
$G(I)$ is Cohen-Macaulay.

\begin{example}\cite[Example 3.10]{cpr}\label{neg}
 Let $(R,\m)$ be the 2-dimensional local Cohen-Macaulay ring
 $$k[|x,y,z,u,v|]/(z^2,zu,zv,uv,yz-u^3,xz-v^3)$$ with $k$ a field and
$x,y,z,u,v$ indeterminate. One can see that the depth $G(\m)=0$ and
 $$F_{\m}(t)=\dfrac{1+3t+3t^3-t^4}{(1-t)^2}.$$
 In this case $e_2(\m)=e_1(\m)-e_0(\m)+1$ but $G(\m)$ is not Cohen-Macaulay.

\end{example}

 \section{Study of $\ov e_3(I)$ }
\noindent So far we have seen that $e_1(\mathcal I)$ and $e_2(\mathcal I)$ are
non-negative 
for any $I$-admissible filtration in a Cohen-Macaulay local ring. But the
non-negativity does not hold true for 
$e_3(\mathcal I)$. Note that $e_3(\m)=-1$ in Example \ref{neg}.
 Marley in \cite{m} has given an another example in which $e_3(I)$ is
negative.
 \begin{example}
  Let $R=k[|x,y,z|]$ where $k$ is a field. 
Then the Hilbert polynomial $P_I(x)$ of the ideal   
  $$I=(x^3,y^3,z^3,x^2y,xy^2,yz^2,xyz)$$
  is given by
  $$P_I(x)=27{x+2\choose 3}-18{x+1\choose 2}+4x+1.$$
  Hence $e_3(I)=-1$.
 \end{example}
\noindent However Itoh has proved that $\ov e_3(I)$ is non-negative in a
Cohen-Macaulay analytically unramified
local ring. He has  proposed a conjecture for  the vanishing of $\ov e_3(I)$ in
\cite{i2}. 

\noindent
{\bf Itoh's Conjecture:}
 Let $(R,\m)$ be a analytically unramified Gorenstein local ring of dimension
$d\geq 3$. Then $\ov e_3(I)=0$ if and only if $\ov r(I) \leq 2.$
 
\noindent Itoh has given a solution to the conjecture when $\ov I=\m.$ In order
to prove this first we recall some preliminary results.

\begin{proposition} \cite[Proposition 10]{i3}\label{ipro1}
 Let $(R,\m)$ be a $d$-dimensional analytically unramified Cohen-Macaulay local
ring and let $I$ be a parameter ideal. Then for all $n\geq 0$,
 \begin{eqnarray*}
\lm(R/\ov{I^{n+1}}) 
&\leq & \lm(R/I){n+d\choose d}-[\lm(\ov I/I)+\lm(\ov{I^2}/I\ov I)]{n+d-1\choose
d-1}\\ &+&\lm(\ov{I^2}/I\ov I){n+d-2\choose d-2}.
\end{eqnarray*}
Moreover, equality holds for all $n \geq 1$ if and only if $\ov r(I) \leq 2.$ 
In particular if $\ov r(I) \leq 2$ then $\ov e_i(I)=0$ for $i\geq 3$.
\end{proposition}

\begin{proof}
 If $n=0$ then the inequality trivially holds. So assume $n\geq 1$. Since
$I^n\ov{I}\subseteq I^{n-1}\ov{I^2}\subseteq \ov{I^{n+1}}$ we have
$\lm(R/\ov{I^{n+1}})\leq 
\lm(R/I^n\ov I)-\lm(I^{n-1}\ov{I^2}/I^n\ov I)$. Since $I^n/I^{n+1}\otimes R/\ov
I\cong I^n/I^n\ov I $  and
 $I^n/I^{n+1}$ is a free $R/I$ module we get 
\begin{eqnarray*}
\lm(R/I^n\ov I)& =& \lm(R/I^n)+\lm(I^n/I^n\ov I)\\
&=& \lm(R/I){n-1+d\choose d}+\lm(R/\ov I){n+d-1\choose d-1}\\
&=&\lm(R/I){n+d\choose d}-\lm(\ov I/I){n+d-1\choose d-1}.
\end{eqnarray*}
Therefore it is sufficient to prove that $\lm(I^{n-1}\ov{I^2}/I^n\ov
I)=\lm(\ov{I^2}/I\ov I){n-1+d-1\choose d-1}$. Since $I^n\ov I\subseteq
I^{n-1}\ov{I^2}\cap I^n\subseteq
 \ov{I^{n+1}}\cap I^n=I^n\ov I$ (by Theorem \ref{ithm2}), we have $I^n\ov
I=I^{n-1}\ov{I^2}\cap I^n$ and hence $I^{n-1}\ov{I^2}/I^n\ov I\cong (I^{n-1}\ov
{I^2}+I^n)/I^n$. \\
Now we prove that the canonical morphism $$I^{n-1}/I^n\otimes \ov{I^2}/I\ov
I\longrightarrow (I^{n-1}\ov {I^2}+I^n)/I^n$$  is an isomorphism. Let
$x_1,\ldots ,x_d$ be a
regular sequence such that $I=(x_1,\ldots ,x_d)$ and let $\{M_j\}$  be the set
of  monomials in $x_1,\ldots ,x_d$ of degree $n-1$. It is sufficient to show
that if
 $\sum_j a_jM_j\in I^n$ with $a_j\in \ov {I^2}$ then $a_j\in I$ and hence
$a_j\in \ov{I^2}\cap I=I\ov I$. Suppose that $\sum_j a_jM_j=\sum_jb_jM_j$ with
$b_j\in I$. 
Since $x_1,\ldots ,x_d$ is a regular sequence we have $a_j-b_j\in I$ for each
$j$ and hence $a_j\in I$. Therefore $(I^{n-1}\ov {I^2}+I^n)/I^n \cong 
I^{n-1}/I^n\otimes \ov{I^2}/I\ov I$.
\end{proof}

\noindent Let $\underbar{x}=x_1,\ldots ,x_d$ be a minimal generators of $I$.
Then the  natural exact sequence of \v{C}ech-complexes,
$$0\longrightarrow C(\underbar x;R)[t,t^{-1}]\longrightarrow C(t^{-1},\underbar
{xt};\ov{\mathcal{R}}(I))\longrightarrow
C(\underbar{xt};\ov{\mathcal{R}}(I))\longrightarrow 0$$
gives an exact sequence
$$0\longrightarrow H^d_{\mathcal{M}}(\ov{\mathcal{R}}(I))\longrightarrow
H^d_{\mathcal{R_+}}(\ov{\mathcal{R}}(I))\longrightarrow
H^d_{\m}(R)[t,t^{-1}]\longrightarrow H^{d+1}_{\mathcal M}
(\ov{\mathcal{R}}(I))\longrightarrow 0.$$

\noindent Consider the the canonical graded 
homomorphism $$\alpha:H^d_{\mathcal{R_+}}(\ov{\mathcal{R}}(I))\longrightarrow
H^d_{\m}(R)[t,t^{-1}].$$
We denote by $\alpha_j$ the graded part of degree $j$ of $\alpha$. Then we have
\begin{lemma}{\rm \cite[Lemma 18]{i2}}\label{ilema}
 The map $\alpha_j=0$ if and only if for all $n \geq 0 ,$ $$
\ov{I^{n+d-1+j}}\subseteq I^n.$$
\end{lemma} 

\begin{proof}
 Let $x_1,\ldots,x_d$ be a system of minimal generators of $I$. We have a
natural morphism of \v{C}ech complexes and its 
cohomologies:
$$\diagram
[\prod_i \ov{\mathcal{R}}(I)_{x_1t \ldots (x_it)^{\wedge}\ldots x_dt}]_j \rto
\dto & [\ov{\mathcal{R}}(I)_{x_1t\ldots x_dt}]_j \rto \dto^{\beta_j}  &
[H^d_{\mathcal R_+}(\ov{\mathcal{R}}(I))]_j \rto \dto^{\alpha_j} & 0 &\\
\prod_i R_{x_1\ldots(x_i)^{\wedge} \ldots x_d} \rto^{f}  &  R_{x_1 \ldots x_d}
\rto &  H^d_{\m}(R)  \rto & 0 & 
\enddiagram$$
Let $\ov{I^{n+d-1+j}} \subseteq I^n$ for all $n \geq 0$. Let
$$z=at^jt^{nd}/(x_1t\ldots x_dt)^n \in 
[\ov{\mathcal{R}}(I)_{x_1t\ldots x_dt}]_j.$$ 
Therefore $a \in \ov{I^{nd+j}} \subseteq I^{nd+j-(d-1+j)}=I^{nd-d+1}$. Let
$x^c=x_1^{c_1}\ldots x_d^{c_d}$ for $(c_1,\ldots,c_d)\in \mathbb N^d$ and
$$a=\sum_{c_1+\cdots+c_d=nd+d-1}r_{c}{x}^{c}.$$ If $c_i < n$ for all $i$, then
$c_1+\cdots+c_d \leq d(n-1)<nd-d+1$. Therefore $c_i \geq n$ for some $i$. Hence 
$a \in (x_1^n,\cdots,x_d^n)$. Let $a=\displaystyle \sum_{i=1}^d r_ix_i^{n}$.
Then $$f\left(\sum_{i=1}^d (-1)^{i-1}r_i/(x_1\ldots \hat{x_i} \ldots
x_d)^n\right)= a/(x_1\ldots x_d)^n.$$ 
Thus $\beta_j(z) = a/(x_1 \cdots x_d)^n \in \im f$ and hence $\alpha_j \equiv
0$.
Conversely, let $\alpha_j \equiv 0$. Since 
$$I^n=\displaystyle \bigcap_{c_1+\cdots +c_d=n+d-1}(x_1^{c_1},\ldots,x_d^{c_d}),
$$
it is sufficient to show that $\ov{I^{n+d-1+j}}$ is contained in
$(x_1^{c_1},\ldots,x_d^{c_d})$ for all integers $c_i \geq 1$ with $c_1+\cdots+
c_d = n+d-1$. Let $a \in \ov{I^{n+d-1+j}}$. Then $z=at^j/(x_1^{c_1}\ldots
x_d^{c_d}) \in [\ov{\mathcal
{R}}(I)_{x_1t\cdots x_dt}]_j
$. Therefore $\beta_j(z) \in \im f$. Let $$\beta_j(z)=a/(x_1^{c_1}\ldots
x_d^{c_d})=\sum (a_ix_i^s)/(x_1\ldots x_d)^s$$ 
for some $a_i \in R$ and $s \geq $ max$\{
c_i \mid i=1,\ldots,d\}.$ Since  $ax_1^{s-c_1}\cdots x_d^{s-c_d}=\sum a_ix_i^s$
we get $a \in (x_1^{c_1},\cdots,x_d^{c_d}).$ 
This proves the assertion. 
\end{proof}

\begin{proposition}{\rm \cite[Proposition 19]{i2}}{\label{ipro2}}
 Let $(R,\m)$ be a Gorenstein, analytically unramified local ring of positive
dimension. Let $I$ be a parameter ideal such that $\ov {I^2}\not=I\ov I$. If
$\ov{I^{n+2}}$ and $\m\ov{I^{n+1}}$ are 
contained in $I^n$ for every $n\geq 0$, then $\lm(\ov{ I^2}/I\ov I)=1$
and $\ov r(I)\leq 2$.
\end{proposition}
\begin{proof}
Let $I=(x_1,\ldots, x_d)$. Since $I^n \supseteq \m \ov{I^{n+1}}$, using
Huneke-Itoh Intersection Theorem $$(I^n:\m)/I^n \supseteq (\ov{I^{n+1}}+I^n)/I^n
= \ov{I^{n+1}}/I^n \cap \ov{I^{n+1}}=
\ov{I^{n+1}}/I^n\ov{I}.$$ Therefore $\lm(\ov{I^{n+1}}/I^n\ov{I}) \leq
\lm((I^n:\m)/I^n)$. As 
$R$ is Gorenstein,
$(x_1^{c_1},\ldots,x_d^{c_d})$ 
is an irreducible ideal. Hence
$$ I^n = \bigcap_{c_1+\cdots+c_d=n+d-1} (x_1^{c_1},\ldots,x_d^{c_d})$$
is an irredundant decomposition of $I^n$ as a product of irreducible ideals
where $c_1, c_2, \ldots, c_d \geq 1.$
The  dimension of $(I^n:\m)/I^n$ as $R/\m$-vector
space equals the number of irreducible 
components of $I^n$. 
Hence $\lm((I^n:\m)/I^n)= {n-1+d-1 \choose d-1}$. Therefore
$\lm(\ov{I^{n+1}}/I^n\ov{I})\leq {n-1+d-1 \choose d-1}.$ Now 
\begin{eqnarray*}
 \lm(R/\ov{I^{n+1}})&=& \lm(R/I^{n+1})-\lm(\ov{I^{n+1}}/I^{n+1})\\
&= &\lm(R/I^{n+1})-\lm(I^n\ov{I}/I^{n+1})-\lm(\ov{I^{n+1}}/I^n\ov{I}).
\end{eqnarray*}
But 
\begin{eqnarray*}
 \lm(I^n\ov{I}/I^{n+1})&=&\lm(R/I^{n+1})-\lm(R/I^n\ov{I})\\
&=& \lm(R/I^{n+1})-\lm(R/I^n)-\lm(I^n/I^n\ov{I})\\
&=& \lm(I^n/I^{n+1})-\lm(I^n/I^n\ov{I}). 
\end{eqnarray*}  
Since $I^n/I^{n+1} \otimes R/\ov{I} \cong I^n/I^n\ov{I},$ for all $n \geq 0,$
$$\lm(I^n/I^n\ov{I})=\lm(R/\ov{I}){n+d-1 \choose d-1}.$$  Hence 
$\lm(I^n\ov{I}/I^{n+1})=\lm(\ov{I}/I){n+d-1 \choose d-1}$. Therefore
\begin{eqnarray*}
 \lm(R/\ov{I^{n+1}})&\geq& \lm(R/I){n+d \choose d} -\lm(\ov{I}/I){n+d-1 \choose
d-1} -{n-1+d-1 \choose d-1}\\
&=& \lm(R/I){n+d \choose d} -[\lm(\ov{I}/I) +1]{n+d-1 \choose d-1}+{n+d-2
\choose d-2}
\end{eqnarray*}
By Proposition \ref{ipro1}, for all $n \geq 0,$
\begin{eqnarray*}
\lm(R/\ov{I^{n+1}}) &\leq& \lm(R/I) {n+d \choose
d}-[\lm(\ov{I}/I)+\lm(\ov{I^2}/I\ov{I})]{n+d-1 \choose d-1} \\ & +&
\lm(\ov{I^2}/I\ov{I}) {n+d-2 \choose d-2}. 
\end{eqnarray*}
Hence $-{n+d-2 \choose d-1} \leq - \lm(\ov{I^2}/I\ov{I}){n+d-2 \choose d-1}$.
Therefore $\lm(\ov{I^2}/I\ov{I})=1$ and equality 
holds in Proposition \ref{ipro1}. Hence $\ov r(I)\leq 2$.

\end{proof}

\begin{theorem} \cite[Theorem 3]{i2}
 Let $(R,\m)$ be a $d$-dimensional analytically unramified Cohen-Macaulay local 
ring and let $I$ be a parameter ideal. Suppose $d\geq 3$. Then \\

{\rm (1) } $\ov e_3(I)\geq 0$ and if $\ov e_3(I)=0$ then $\ov{I^{n+2}}\subseteq
I^n$ for every $n\geq 0$.\\
{\rm (2)} If $R$ is Gorenstein and $\ov I=\m$ then $\ov e_3(I)=0$ if and only if
$\ov{r}(I) \leq 2.$  
\end{theorem}

\begin{proof}
(1) Apply induction on $d$. Assume $d=3$. By Theorem 
\ref{blanc} and Theorem \ref{s2}, we have $\ov e_3(I)=
\lm(H^3_{\mathcal R_+}(\ov{\mathcal{ R}}(I))_0\geq 0$. If $\ov e_3(I)=0$ then
$\left[H^3_{\mathcal R_+}(\ov{\mathcal{ R}}(I))\right]_0=0$. 
Therefore $\alpha_0=0$ in lemma \ref{ilema}. Hence $\ov{I^{n+2}}\subseteq I^n$
for every $n\geq 0$.
Using Theorem \ref{s1} it is easy to see that the result holds in higher
dimension also.\\
(2) Let $R$ be Gorenstein and $\ov I=\m$. Let $\ov{e}_3(I)=0$. Then
$\m\ov{I^{n+1}}\subseteq \ov{I^{n+2}} \subseteq I^n$ for 
all $n \geq 0$. By proposition \ref{ipro2}, $\ov{I^{n+2}}=I^n\ov {I^2}$ for
every $n\geq 0$. The converse
follows from Proposition \ref{ipro1}.
\end{proof}

\noindent Huckaba-Huneke in \cite{hh}  gave a different proof of the
non-negativity of $\ov e_3(I)$
by reducing to dimension three and then applying the following theorem.
\begin{theorem}{\rm \cite[Theorem 3.1]{hh}}
 Let $(R,\m)$ a $d$-dimensional Cohen-Macaulay local ring. Let $I$ be a normal
ideal with $\grade(I)\geq 2$ and $I$ be integral over an ideal generated by an
$R$-regular
sequence. Then there exists $n$ such that $\depth G(I^n)\geq 2$.
\end{theorem}
\noindent In \cite{cpr} Corso, Polini and Rossi showed that in
 an analytically unramified Cohen-Macaulay local ring of dimension three  if
$\ov e_3(I)=0$ for an $\m$-primary ideal then $G(I^n)$ is Cohen-Macaulay for
large $n$.
\begin{theorem} \cite[Corollary 4.5]{cpr}
 Let $(R,\m)$ be a local Cohen-Macaulay ring of dimension three with infinite 
residue field. Let $I$ be an $\m$-primary ideal of $R$ such that $I$ is 
asymptotically normal. Then $e_3(I)=0$ if and only if there exists some $n$
such 
that the reduction number of $I^n$ is at most two. Under these conditions,
$G(I^n)$ is Cohen-Macaulay for $n>>0$.
\end{theorem}
The following example from \cite{cpr} shows that if an ideal $I$ is
asymptotically normal such that $e_3(I)=0$ then $G(I)$ need not be
Cohen-Macaulay.
\begin{example}
 Let $R=k[|x,y,z|]$ with $k$ a field and $x,y,z$ indeterminates. Consider the
ideal $I=(x^2-y^2,y^2-z^2,xy,xz,yz)$. Then $I^n$ is integrally closed for every
$n\geq 2$. We also have
 $$F_I(t)=\dfrac{5+6t^2-4t^3+t^4}{(1-t)^3}$$
 which gives $e_3(I)=0$ but $G(I)$ has depth zero. On the other hand $I^2$ is a
normal ideal with $e_3(I^2)=0$ and $G(I^2)$ is Cohen-Macaulay and reduction
number is two. 
\end{example}
 Huckaba and Huneke in \cite{hh} have  shown that in a two dimensional
analytically unramified Cohen-Macaulay local ring, for some $n,$ 
$\ov G({I^n})$ is Cohen-Macaulay.
 \begin{theorem} \cite[Theorem 3.1]{hh}
  Let $(R,\m)$ be a two dimensional Cohen-Macaulay local ring with infinite
residue field and let $I$ be an normal $\m$-primary ideal of $R$. Then there
exists $n$ such that $G(I^n)$ is Cohen-Macaulay. 
 \end{theorem}
 They have also shown that the above result cannot be extended to higher
dimension in view of
 \begin{example}
  Let $k$ be a field of characteristic $\not= 3$. Set $R=k[|x,y,z|]$. Let 
  $$N=(x^4,x(y^3+z^3),y(y^3+z^3),z(y^3+z^3))$$
  and set $I=N+\m^5$ where $\m=(x,y,z)$. Then $I$ is a height $3$ normal ideal
of $R$ and $G(I^n)$ is not Cohen-Macaulay for any $n\geq 1$.
 \end{example}

  \section{Normal Hilbert Polynomials in two dimensional regular local rings}
  {\it Throughout this section, we assume that  $(R,\m)$
  is a regular local ring of dimension two unless otherwise stated.}
  
  In this section we present a formula for the normal Hilbert polynomial of two
$\m$-primary ideals  of $R.$ This is a consequence of a result of Lipman and
Teissier \cite {lt} that for any complete $\m$-primary ideal in $R,$ $\ov r(I)
\leq 1.$ The formula has also been derived in \cite{lip}, \cite{hun2} and
\cite{r2}.      We shall use joint reductions and Zariski's theory of complete
ideals in $R$
  to derive this formula.  We shall also derive the same formula using a formula
of Hoskin and Deligne for the colength of a complete $\m$-primary ideal.
  
  A crucial step for obtaining a formula for the normal Hilbert polynomial of
two ideals  is to show that if $I$ and $J$ are complete $\m$-primary ideals in
$R$ with $R/\m$ infinite then there exist $a\in I$ and $b\in J$ such that
$aI+bJ=IJ$. This was proved in
  \cite{verma}. By taking $I=J$ we get Lipman-Teissier formula: $I^2=(a,b)I.$  
  
Joint reductions were introduced by Rees in \cite{r}. Let $I_1,I_2\ldots, I_r$
be ideals in a local ring $(R,\m)$. A set of elements $(x_1,\ldots ,x_r)$ such
that $x_i\in I_i$ is called a joint reduction of the set of ideals $(I_1,\ldots
, I_r)$ if there are positive integers $a_1,\ldots ,a_r$ such that 
  $$\sum_{i=1}^r x_iI_1^{a_1}I_2^{a_2}\ldots I_i^{a_i-1} \ldots
I_r^{a_r}=I_1^{a_1} \ldots I_r^{a_r}.$$
  
  Rees showed \cite{r} that if $R/\m$ is infinite and $I_1,\ldots, I_r$ are
$\m$-primary ideals in any local ring $(R,\m)$ then joint reductions exist. Liam
O'Carroll showed the existence of joint reductions for the arbitrary ideals
\cite{o}. 
  
  We now recall a few facts from Zariski's theory of complete ideals.  An ideal
$I$ of $R$ is called {\bf contracted} if there is an $x\in \m\setminus \m^2$
such that $IR[\m/x]\cap R=I$. Here $R[\m/x]=R[y/x]$  where $\m=(x,y)$. By 
\cite[Lemma 3, Appendix 5]{zs}, if $I$ and $J$ are contracted from $R[\m/x]$
then so is $IJ$. Rees \cite[Lemma 3.1]{r2} and Lipman \cite[Corollary 3.2]{lip}
showed that if $I$ is $\m$-primary and contracted then $\mu(I)=1+o(I)$ where
$\mu(I)=\dim(I/\m I)$ and $o(I)=\m$-adic order of $I=\max\{n \mid I\subseteq
\m^n\}$. Huneke and Sally \cite[Theorem 2.1]{hunsally} proved that if $R/\m$ is
infinite and $\mu(I)=1+o(I)$ then $I$ is a contracted ideal. An important fact
about complete ideal is that they are contracted \cite[Proposition 3.1]{hun2}. 
  
  Let $I$ be an $\m$-primary ideal contracted from $S=R[\m/x]$. Let $N$ be a
maximal ideal containing $\m S$. Then $S_N$ is a $2$-dimensional regular local
ring. Suppose $o(I)=r$. Then $IS=x^rI'S$ for an ideal $I'$ of $S$. We say $I'$
is a transform of $I$ in $S$ and $I'_N$ is called the transform of $I$ in $S_N$.
By Proposition 5 of \cite[Appendix 5]{zs}, if $I$ is complete then so is $I'$
and $I'_N$. Finally by \cite[Proposition 3.6]{hun2}, $e(I)>e(I'_N)$.
  
  \begin{theorem} \cite[Theorem 2.1]{verma}
  Let $(R,\m)$ be a $2$-dimensional regular local ring with $R/\m$ infinite. 
Let $I$ and $J$ be $\m$-primary complete ideals. Then there exists  a  joint
reduction $(a,b)$ of $(I,J)$ such that 
  $$aJ+bI=IJ.$$
  \end{theorem}
  
  \begin{proof}
  We may assume that $R/\m$ is infinite. Apply induction on
$t=\max\{e(I),e(J)\}$. Let $t=1.$ Then $e(I)=1=\lm(R/L)$ where $L$ is a minimal
reduction of $I.$  Thus $I=L=\m$ and similarly $J=\m$. Since $R$ is regular
$\m=(x,y)$ for some $x,y\in \m$. Thus $x\m+y\m=\m^2$. Now let $t>1$. Let $(a,b)$
be a joint reduction of $(I,J)$. We show that the ideal $K=aJ+bI$ is contracted.
Let 
  $$aI^{n-1}J^n+bI^nJ^{n-1}=I^nJ^n$$
  for some $n$. This equation implies that $a$ ( respectively $b$) is part of a
minimal basis of $I$ (respectively $J$). Let $o(I)=r$ and $o(J)=s$. Since $I$
and $J$ are complete, they are contracted. Hence $\mu(I)=r+1$ and $\mu(J)=s+1$.
Let $I=(a_0,\ldots ,a_r)$ and $J=(b_0,\ldots, b_s)$ where $a=a_0$ and $b=b_0$.
Now we show that $K$ is minimally generated by the $r+s+1$ elements:
  $$a_0b_0, a_0b_1, \ldots ,a_0b_s, b_0a_1,b_0a_2, \ldots ,b_0a_r.$$
  Indeed, let 
  $$\sum_{i=0}^sa_0b_iu_i+\sum_{j=1}^rb_0a_jv_j=0$$
  where $u_0,\ldots ,u_s$ and $v_1,\ldots ,v_r\in R.$ As $a,b$ is a regular
sequence there is an $f\in R$ such that 
  $$a_0f=\sum_{j=1}^ra_jv_j \mbox{ and hence } b_0f=-\sum_{i=0}^sb_iu_i.$$
  Hence $f,v_1,\ldots, v_r\in \m$ and $u_0+f,u_1,\ldots ,u_s\in \m$. Thus
$\mu(K)=r+s+1$. Since $K=aJ+bI$ is a reduction of $IJ$, $o(K)=o(IJ)=r+s$. Hence
$K$ is a contracted ideal. Pick $x\in \m \setminus \m^2$ such that $K$ and $IJ$
are contracted from $S=R[\m/x]$. We write $a=a'x^r$, $b=b'x^s$, $I=x^{r}I'S$,
$J=x^sJ'S$ where $I',J'$ are ideals in $S$ and $a',b'\in S$. Then 
  $$KS=(a'J'+b'I')x^{r+s}S, ~~ IJS=x^{r+s}I'J'S.$$
  Therefore $(a',b')$ is a joint reduction of $(I',J')$. As $IJ$ and $K$ are
contracted from $S$ it is enough to show $a'J'+b'I'=I'J'$. To prove that 
  $a'J'+b'I'=I'J'$, localize at any maximal ideal $N$ of $S$ containing
$a'J'+b'I'$. Since $e(I'_N)<e(I)$ and $e(J'_N)<e(J)$, it follows that
$(a'J'+b'I')S_N=I'J'S_N$. Therefore $a'J'+b'I'=I'J'$ and consequently
$IJ=aJ+bI$.

  \begin{lemma}
  Let $(R,\m)$ be a local ring of dimension $\geq 2$. Let $I$ and $J$ be ideals
of $R$ and $a\in I$, $b\in J$ be such that $(a,b)$ is a regular sequence. Then
the $R$-module homomorphism
  $$f: \frac{R}{I}\oplus \frac{R}{J}\longrightarrow \frac{(a,b)}{aJ+bI}$$
  defined as $f(\ov x, \ov y)=(xb+ya)+(aJ+bI)$ is an isomorphism.
\end{lemma}

\begin{proof}
It is clear that $f$ is surjective. For injectivity let $xb+ya\in aJ+bI$. Then
$xb\in (a,bI)$. Choose $c\in R$ and $d\in I$ such that $xb=ac+bd$. Hence
$b(x-d)=ca$ which implies $x-d\in (a:b)=(a)$. Thus $x\in I$. Similarly $y\in J$.
Therefore $f$ is injective and hence an isomorphism.
\end{proof}

\begin{theorem}
Let $(R,\m) $ be a $2$-dimensional Cohen-Macaulay local ring with $R/\m$ infinite.
 Let $I$ and $J$ be
$\m$-primary ideals. Then the following are equivalent. 
\begin{enumerate}
\item[{\rm (a)}] There exist $a\in I$ and $b\in J$  such that $aJ+bI=IJ$,
\item[{\rm (b)}] For all $r,s\geq 1$ $a^rJ^s+b^sI^r=I^rJ^s,$
\item[{\rm (c)}] For all $r,s\geq 1$, $\lm(R/I^rJ^s)=\lm(R/I^r)+rs
e_1(I|J)+\lm(R/J^s)$ where $e_1(I\mid J)=e(a,b)$,
\item[{\rm (d)}] $e_1(I|J)=\lm(R/IJ)-\lm(R/I)-\lm(R/J).$
\end{enumerate}
\end{theorem}

\begin{proof}
$(a)\implies (b)$ By symmetry it is enough to show that $a^rJ+bI^r=I^rJ$ for all
$r\geq 1$, Apply induction on $r$.  Assume that $a^rJ+bI^r=I^rJ$. Then
\begin{eqnarray*}
I^{r+1}J &=& I(I^rJ)\\
&=& I(a^rJ+bI^r)\\
&=& a^rIJ+bI^{r+1}\\
&=& a^r(aJ+bI)+bI^{r+1}\\
&=& a^{r+1}J+bI^{r+1}.
\end{eqnarray*}
$(b)\implies (c)$ Since 
$$R/I^r\oplus R/J^s\cong (a^r,b^s)/(a^rJ^s+b^sI^r)=(a^r,b^s)/I^rJ^s,$$
we get
\begin{eqnarray*}
\lm(R/I^rJ^s) &= &\lm(R/(a^r,b^s))+\lm(R/I^r)+\lm(R/J^s)\\
&=& rs \lm(R/(a,b))+\lm(R/I^r)+\lm(R/J^s)\\
&=& rs e_1(I|J)+\lm(R/I^r)+\lm(R/J^s).
\end{eqnarray*} 
$(c)\implies (d)$ Clear.\\
$(d)\implies (a)$ As $R/\m$ is infinite there exists a joint reduction $(a,b)$
of $(I,J)$. Since 
$$R/I\oplus R/J\cong (a,b)/(aJ+bI),$$
we have $\lm(R/aJ+bI)-e_1(I|J)=\lm(R/I)+\lm(R/J)$. Substitute
$e_1(I|J)=\lm(R/IJ)-\lm(R/I)-\lm(R/J)$ to get $IJ=aJ+bI.$
\end{proof}

\begin{corollary}\cite{lip} \label{lip}
Let $I$ be an $\m$-primary ideal of a regular local ring of dimension $2$. Then
for all $n\geq 1$, 
$$\lm(R/\ov{I^n})=e(I){n+1\choose 2}-[e(I)-\lm(R/\ov I)]n.$$
\end{corollary}

\begin{proof}
By Zariski's theorem, $\ov{I^n}=\ov I^n$ for all $n$. Hence we may assume that
$I$ is a complete ideal. Use induction on $n$. Assuming the result of $n-1$, we
have 
\begin{eqnarray*}
\lm(R/I^n) &=& \lm(R/I^{n-1}I)= \lm(R/I^{n-1})+(n-1)e(I)+\lm(R/I)\\
&=& e(I){n\choose 2}-[e(I)-\lm(R/I)](n-1)+n-e(I)+\lm(R/I)\\
&=& e(I){n+1\choose 2}-[e(I)-\lm(R/ I)]n.
\end{eqnarray*}
  \end{proof}
  
\begin{corollary}
Let $I$ and $J$ be $\m$-primary ideals of a regular local ring of dimension $2$.
Then for all $r,s\geq 0$, 
$$\lm(R/\ov{I^rJ^s})=e(I){r \choose 2}+rs e_1(I|J)+ e(J){s\choose 2}+r\lm(R/\ov
I)+s\lm(R/\ov J).$$
\end{corollary}

\begin{proof}
We may assume that $R/\m$ is infinite. Thus for any joint reduction $(a,b)$ of
$(\ov I, \ov J)$, $a\ov J+b \ov I=\ov I \ov J.$ Thus 
$$\lm(R/\ov I^r\ov J^s)=\lm(R/\ov I^r)+rs e_1(I|J)+\lm(R/\ov J^s).$$
By Zariski's theorem $\ov I^r=\ov {I^r},$ $\ov J^s=\ov {J^s}$ and $\ov I^r \ov
J^s=\ov{I^rJ^s}$. Now use Corollary
\ref{lip} to finish the proof.
\end{proof}

  We end this section by sketching an alternate proof of Lipman's formula for
normal Hilbert polynomial of an $\m$-primary ideal in a regular local ring of
dimension two. This is done by invoking a formula of Hoskin \cite{hos} which is
also proved by Deligne \cite{del} and Rees \cite{r2} independently. We refer the
reader to section 14.5 of \cite{hunsally} for a very readable account.
  
  Let $I$ be an $\m$-primary complete ideal in a $2$-dimensional regular local
ring $(R,\m)$. We say that a $2$-dimensional regular local ring $(S,\n)$
dominate $R$ birationally if $R\subset S$, $\n\cap R=\m$ and $R$ and $S$ have
equal fraction fields. Let $N$ be a maximal ideal of $T=R[\m/x]$ where $x\in
\m\setminus \m^2$ and $\m T\subset N$. Then $T_N$ is a $2$-dimensional regular
local ring, called a local quadratic transform of $R$. Abhyankar \cite[Theorem
3]{a} showed that if $S$ birationally dominates $R$ then there is a unique
sequence 
  $$R=R_0\subset R_1\subset \ldots \subset R_n=S$$
  of $2$-dimensional regular local ring such that $R_i$ is a local quadratic
transform of $R$ for $i=1,\ldots, n$. A point basis of $I$ is the set 
  $$\mathcal B(I)=\{o(I^T): T \mbox{ dominates } R \mbox{ birationally }\}$$
  where $I^T=$ transform of $I$ in $T$ and $o(I^T)=\m_T$-adic order of $I^T$.
  \end{proof}
  
  \begin{theorem}[Hoskin-Deligne Formula]
Let $(R,\m)$ be a regular local ring of dimension $2$ with infinite residue
field and let $I$ be a complete $\m$-primary ideal of $R$. Then 
$$\lm(R/I)=\sum_{T\succ R}{o(I^T)+1 \choose 2}[T/\m_T :R/\m]$$
where the sum is over all $2$-dimensional regular local rings $T$ which
birationally dominate $R$, written as $T\succ R$ and $\m_T$ is the maximal ideal
of $T$.
  \end{theorem}
  
\begin{corollary}
Let $I$ be a complete $\m$-primary ideal of $2$-dimensional regular local ring
$(R,\m)$. Then for all $n\geq 1$,
$$\lm(R/I^n)=e(I){n+1\choose 2}-[e(I)-\lm(R/I)]n.$$
\end{corollary}

\begin{proof}
By Zariski's theorem $I^n$ is complete. Moreover $o((I^n)^T)=n o(I^T)$. Hence
for all $n\geq 1$,
$$\lm(R/I^n)=\sum_{T\succ R}{n o(I^T)+1 \choose 2}[T/\m_T :R/\m].$$
Let us write $[T/\m_T :R/\m]=d(T)$ and $o(I^T)=o(T)$. Then for all $n \geq 0$,
\begin{eqnarray*}
\lm(R/I^n) &=& \sum_{T\succ R}\frac{1}{2}(n^2 o(T)^2+no(T))d(T)\\
&=& \frac{1}{2}e(I)(n^2+n)-e_1(I)n+e_2(I)
\end{eqnarray*}
Hence $$e(I)=\displaystyle{\sum_{T\succ R}o(T)^2d(T)},
\;\;e_1(I)=\displaystyle{\sum_{T\succ R}{o(T)\choose 2}d(T)},
\;\; e_2(I)=0.$$
These expressions  imply that
$e_1(I)=e(I) - \lm(R/I).$
\end{proof}

The Hoskin-Deligne formula has been generalized for finitely supported 
$\m$-primary ideals in regular local rings of dimension at least three by
C. D'Cruz \cite{cl}. B. Johnston \cite {j} established
a multiplicity formula for the same class of ideals.
   
   \section{The normal Hilbert polynomial of a monomial ideal}
   Let $R=k[x_1,x_2,\ldots, x_d]$ be the polynomial ring over a field $k$ with 
   dimension $d\geq 2. $  The maximal homogeneous ideal of $R$ will be denoted
by $\m.$ Let $I=(x^{v_1},\ldots ,x^{v_q})$ be an $\m$-primary 
   monomial ideal where $v_i\in \mathbb N^d$ for $i=1,2,\ldots, q.$ If 
$w=(w_1,\ldots ,w_d)\in \mathbb N^d$ then we put $x^w=x_1^{w_1}\ldots
x_d^{w_d}.$ First we describe the integral closure of $I$ in terms of convex 
   polytopes which will lead to a formula for  the normal Hilbert polynomial of
$I$ in terms of Ehrhart polynomials of certain polytopes derived from the
exponent vectors $v_1, v_2, \ldots, v_q.$
   
  Let $e_1, e_2, \ldots, e_d$ be the standard basis vectors of $\mathbb Q^d.$ Since $I$
is an $\m$-primary ideal, there are natural numbers $a_1, a_2, \ldots,a_d$ such
that $v_i=a_ie_i$ for $i=1,2, \ldots,d$ after we have permuted the generators of
$I.$  Let ${a}=(1/a_1,1/a_2, 1/a_3, \ldots, 1/a_d).$
We may asume that  $\langle v_i, a\rangle < 1$ for $i=d+1, \ldots, s$ and
$\langle v_i, a\rangle \geq 1$ for 
   $i=s+1,  \ldots,q.$  Consider  the convex polytopes in $\mathbb Q^d.$
     \begin{eqnarray*}
   P &=& \text{conv}(v_1,v_2,\ldots,v_s) \\
   S &=& \text{conv}(0,v_1, v_2, \ldots,v_d)
   \end{eqnarray*}
   and the convex polyhedron $Q=\mathbb Q_{+}^d+\text{conv}(v_1, v_2,
\ldots,v_q).$
  B. Teissier \cite{Te} identified $\ov{I^n}$ in terms of lattice points of
$nQ$.  See also \cite{vil}. 
   \begin{proposition}\label{vill}
   With above settings  $\ov{I^n}=(\{x^a\mid a\in nQ\cap \mathbb Z^d\})$ for all
$n\in \mathbb N$.
   \end{proposition}
   \begin{proof}
   Let $x^\alpha \in \ov {I^n}$ and $J=I^n$ then $x^\alpha \in \ov J$ which
implies that there exists $0\not= m\in \N$ such that $x^{m\alpha}\in J^m$. We
have $I= (x^{v_1},\ldots, x^{v_q})$. Thus $x^{m\alpha}\in (x^{v_1},\ldots,
x^{v_q})^{mn}$ which implies that $x^{m\alpha}=(x^{v_1})^{k_1}\ldots (x^{v_q})^{k_q}
x^{\beta}$ where $\sum k_i =mn$. Thus we get $m\alpha =\sum k_iv_i +\beta $
which implies $\alpha={\displaystyle \sum }\frac{k_i}{m} v_i +\frac{\beta}{m}$.
Let $l_i=\frac{k_i}{m}$ then $\alpha={\displaystyle\sum }l_iv_i +\beta' $ where
$\sum l_i=n$ and $0\leq\beta'\in \Q^d$. Thus
$\frac{\alpha}{n}={\displaystyle\sum }\frac{l_i}{n} v_i+\frac{\beta'}{n}$ where
${\displaystyle\sum} \frac{l_i}{n}=1$ which implies that $\frac{\alpha}{n}\in
\text{conv}(v_1,\ldots,v_q)+\Q_+^d=Q$. Therefore $\alpha\in nQ\cap \Z^d.$\\ 
   \indent Conversely let $\alpha \in nQ\cap\Z^d$ for some $ n \geq 0.$  Then
$\alpha= u+w$ where $u\in n \text{conv}(v_1,\ldots,v_q)$ and $w\in \Q^{d}_+$.
Therefore $\alpha=n\sum l_iv_i+w$ with $0\leq l_i\leq 1$ and $\sum l_i=1$. Let
$m=\mbox{l.c.m(denominators of }l_i)$ and $m_i=ml_i\in \Z$. Thus
$m\alpha=mw+n\sum m_iv_i$ which implies that
$x^{m\alpha}=x^{mw}((x^{v_1})^{m_1}\ldots (x^{v_q})^{m_q})^n.$ Therefore
$x^{m\alpha}\in I^{mn}$ and hence $x^\alpha\in \ov{I^n}$.
   \end{proof}
   Villarreal \cite{vil} has shown that the normal Hilbert polynomial of a
monomial ideal is the difference of two Ehrhart polynomials. First we recall the
Ehrhart polynomial of a convex polytope and some of its properties.
   \begin{theorem}[Ehrhart, 1962]
    Let $P$ be an integral  convex polytope of dimension $d.$ Then the function 
   $\chi_P(n)=|nP \cap \mathbb Z^d|$
   for $n\in \mathbb N$ is a polynomial function of degree $d$ denoted by
   $$E_P(n)=a_dn^d+\cdots +a_1n+a_0$$ with $a_i\in \mathbb Q$ for all $i$.
   \end{theorem}
   The polynomial $E_P(n)$ is called the {\bf Ehrhart polynomial} of $P$. Some
well known properties of $E_P$ are
   \begin{enumerate}
    \item  Let  $\vol (P)$ denote the relative volume of $P$. Then $a_d=\vol
(P).$ 
    \item  Suppose  $F_1,\ldots F_s$ are facets of $P$. Then $a_{d-1}=\left(
\sum_{i=1}^s\vol(F_i)\right)/2 .$ 
    \item  We have $\chi_P(n)=E_P(n)$ for all integers $n\geq 0$. 
   \end{enumerate}
   Moral\'es first proved the following result in \cite{m2} later Villarreal
\cite{vil} has given another proof. 
   \begin{theorem}\cite[Proposition 3.6]{vil} Let $I$ be an $\m$-primary
monomial ideal of the polynomial ring $k[x_1,x_2, \ldots,x_d]$ over a field $k.$
Then 
    $$\lm(R/\ov{I^n})=|\mathbb N^d \setminus nQ|= E_S(n)-E_P(n)$$ for all $n.$
In particular
   $$\ov P_I(n)=[\vol(S)-\vol(P)]n^d+\text{lower degree terms}$$
   and the constant term of $\ov P_I(x) $ is zero.
   \end{theorem}
   \begin{proof}
   Since $\lm(R/\ov{I^n})=dim_k(R/\ov{I^n})$, by Proposition \ref{vill}, we have 
$\lm(R/\ov{I^n})=|\mathbb N^d \setminus nQ|$. As $E_S(0)=E_P(0)=0$, we get the
equality at $n=0$. Now assume $n\geq 1$. Notice that we can decompose $Q$ as 
$Q=(\Q^d_+\setminus S)\cup P$. Thus we get $$nQ=(\Q^d_+\setminus nS)\cup nP
\Longrightarrow \N^d\setminus nQ=[\N^d\cap (nS)]\setminus [\N^d\cap (nP)].$$
   Hence we get $\lm(R/\ov{I^n})=E_S(n)-E_P(n)$. Using the properties of Ehrhart
polynomial we can write
   $$\ov P_I(n)=[\vol(S)-\vol(P)]n^d+\text{lower degree terms}$$ with the
constant term zero.
   \end{proof}
   
   It follows from a result of Marley \cite{m} that  not only the normal Chern
number but all the coefficients of the normal Hilbert polynomial of a  monomial
ideal in a polynomial ring are non-negative. We  present a different proof of
this theorem.
    \begin{theorem} Let $I$ be a zero-dimensional monomial ideal of a
    polynomial ring $R=k[x_1,x_2,\ldots,x_d]$ over a field $k$. Then
    $\ov{e}_i(I) \geq 0$ for all $i=0,1,2,\ldots, d-1$.
    \end{theorem}
    
    \begin{proof}
      We may assume without loss of generality that $k$ is infinite.  The
integral closure of a monomial ideal is a monomial ideal by \cite[Proposition
1.4.6]{hs}. Hence  The Rees algebra 
    $\ov{\mathcal {R}}= \ov{\mathcal
{R}}({I})=\bigoplus_{n=0}^{\infty}\ov{I^n}t^n$ is a normal semigroup ring. Hence
by Hochster's theorem \cite[Theorem 6.3.5]{bh} $\ov{\mathcal{R}}$ 
    is Cohen-Macaulay. Therefore the associated graded ring $\ov G=\ov
G(I)=\oplus_{n=0}^{\infty}\ov{I^n}/\ov{I^{n+1}}$ is Cohen-Macaulay by
\cite{viet}.  Let $J$ be a minimal reduction of $I$. 
   Then the initial forms of generators of $J$ in degree one component
    of $\ov G$  form a $\ov G$-regular sequence. Hence
     \begin{eqnarray*} 
     H(\ov G/J\ov G,z)&=& (1-z)^d H(\ov G,z)\\
             &=& \lm(R/\ov{I}) + \lm(\ov{I}/J+\ov{I^2})z+ \cdots
+\lm(\ov{I^{d-1}}/J\ov{I^{d-2}}+\ov{I^d})z^{d-1}\\
     &:=& f(z)
    \end{eqnarray*}
    In the above calculation we have used the celebrated theorem of
Brian\c{c}on-Skoda Theorem \cite[ Theorem 13.3.3]{hs} which asserts that 
$\ov{I^{n+d}} \subseteq I^{n+1}$ for all $n \geq 0.$ Since for $i=1,\ldots ,d-1$
$$ i \; ! \;\ov{e}_i(I)=\frac{d^if(z)}{dz^i}\left|\right._{z=1}.$$ 
    Therefore $\ov e_i(I)\geq 0$ for $i=0,1,\ldots ,d-1$.
    \end{proof}

\end{document}